\documentclass{amsart}

\usepackage{graphicx}
\usepackage{subfigure}

\usepackage{amsmath}
\usepackage{amssymb}

\newtheorem{theorem}{Theorem}
\newtheorem{lemma}{Lemma}[section]
\newtheorem{corollary}[lemma]{Corollary}
\newtheorem{proposition}{Proposition}

\numberwithin{equation}{section}

\begin{document}

\title{Points with finite orbits for trace maps}

\author{Stephen  Humphries}
\address{Department of Mathematics,
 Brigham Young University, 
Provo, UT 84602, U.S.A.} 
\email{steve@mathematics.byu.edu}

\subjclass[2000]{Primary 37C25; Secondary 15A30, 37D40, 37C85,  20F36}
\date{\today}

\begin{abstract}
We study an action of ${\rm Aut}(F_n)$ on $\mathbb{R}^{2^n-1}$ by trace
maps, defined using the traces of $n$-tuples of 
matrices in $\mathrm{SL}(2,\mathbb{C})$ having real traces.  
  We determine the finite
orbits for this action. These orbits essentially come from (i)  the finite subgroups of $\mathrm{SL}(2,\mathbb C)$, and (ii) a dense set of (rational) points in an embedded quotient of an $n$-torus. 
\end{abstract}

\maketitle

\setcounter{section}{-1}
\section{Introduction}

Many authors have studied trace maps \cite{BGJ,C,D,hm1,hm2,LPW,P,PWW,RB1,R}, 
which give an action of
${\rm Aut}(F_2)$ on $\mathbb{R}^3$ (here $F_n$ is a rank $n$ free group) 
and form an interaction between representation
theory and dynamical systems.  
Goldman and others have also studied this action in the context
of character varieties, see \cite{Go1,Go2,Go3,Br1,Br2,Tan}.  The case $n=2$ is that usually studied.
Here there are elements  $ \sigma_1,\sigma_2\in {\rm Aut}(F_2)$ 
acting as follows on $\mathbb R^3$: 
\begin{equation}\label{eq0.1}
\sigma_1(x,y,z)=(z,y,2yz-x),\;
 \sigma_2(x,y,z)=(x,2xy-z,y). 
\end{equation}
This comes from an action of ${\rm Aut}(F_2)$ on triples $(x_1,x_2,x_{12})$  of traces
corresponding to  pairs of matrices  $  A_1, A_2 \in 
\mathrm{SL}(2,\mathbb{C})$, where
\begin{equation}\label{eq0.2}
x_1= \mathrm{trace}(A_1)/2,\; x_2=\mathrm{trace}(A_2)/2,\; x_{12}=\mathrm{trace}(A_1A_2)/2,
\end{equation}
are real numbers; see \S \ref{sec1} for more details.

The ${\rm Aut}(F_2)$ action preserves the level surfaces $E_t =
E^{-1}(t)$ of the function
\begin{equation}\label{eq0.3}
E:\mathbb{R}^3 \rightarrow \mathbb{R},\,\,\,
E(x,y,z)= x^2 +y^2 +z^2 -2xyz.
\end{equation}
A point of $E_t$ will be
said to  {\it lie on level $t$}.  Now $E_1$ contains (within the
cube $[-1,1]^3$) a curvilinear tetrahedron that is parameterised as:
\begin{equation}\label{eq0.4}
\partial \mathcal{T} = \{ (\cos 2\pi\theta_1, \cos 2\pi\theta_2,
\cos 2\pi(\theta_1 + \theta_2)): (\theta_1, \theta_2)^T \in
\mathbb{R}^2 \},
\end{equation}
where the action of ${\rm Aut}(F_2)$ is given via the corresponding 
$\mathrm{GL}(2,\mathbb{Z})$-action on $(\theta_1, \theta_2)^T$.

Thus, for $m \in \mathbb{N}$, any $\sigma \in {\rm Aut}(F_2)$ permutes the
points of $\partial \mathcal{T}$ corresponding to $ \{(p/m,q/m)^T:
0\leq p,q <m \}$, which therefore have finite ${\rm Aut}(F_2)$-orbits.

One can generalize the above to the situation where, for $n \ge 2$, the automorphism group ${\rm Aut}(F_n)$ acts on a trace variety corresponding to $n$-tuples of matrices $A_1,A_2,\dots,A_n \in \mathrm{SL}(2,\mathbb C)$ with real traces. In this case the trace variety  is generated by   the $2^n-1$ half-traces
\begin{align*}
& x_I={\rm trace} (A_{i_1}A_{i_2}\dots A_{i_k})/2,\,\, I=(i_1,i_2,\dots,i_k), \\&\qquad  1\le k \le n, 1\le  i_1<i_2<\dots<i_k\le n.\end{align*}
Here `generated' means 
 that for any $A \in \langle A_1,A_2,\dots,A_n\rangle$, the half-trace of $A$ is an integer polynomial in these $x_I$. We note that these $2^n-1$  traces are certainly not a minimal generating set (see \cite {Mag80,Fl} for example), however they are natural in our situation.

Thus we obtain an action of  ${\rm Aut}(F_n)$ on $\mathbb R^{2^n-1}$.

The aim of this paper is to study the set $\mathcal F_n, n=2,3, $ of points of
$\mathbb R^{2^n-1}$ which have finite orbit under the action of
 ${\rm Aut}(F_n)$ and the set $\mathcal P_n$ of points which are
periodic for each element of  ${\rm Aut}(F_n)$.  It is clear that
$\mathcal F_n \subseteq \mathcal P_n$ and the main result of this
paper is to prove  that $\mathcal F_n = \mathcal P_n$ and to completely determine $\mathcal F_n$.

 The  result in the case $n=2$ follows from results of \cite {Dub}, where the authors use this result  to study 
 the global analytic properties of the solutions of a particular
family of Painlev\'e VI equations. 
  Here we note that 
  for any $(x,y,z)^T \in \mathbb R^3\setminus E_1$ there is a pair $(A_1,A_2) \in \mathrm{SL}(2,\mathbb Z)$ (given explicitly in $\S 1$), determined up to conjugacy, and which we say is {\it associated to $(x,y,z)^T$}. The result of 
 \cite {Dub} is:

 \begin{theorem}\label{th1}
Consider the action of  ${\rm Aut}(F_2)$ on $\mathbb{R}^3$.  
Let $\mathcal F=\mathcal F_2$ denote the points that have finite
 ${\rm Aut}(F_2)$-orbit and 
$\mathcal P=\mathcal P_2$ those that have finite orbit under each element of
 ${\rm Aut}(F_2)$.   
Then any point $p \in \mathcal P$ with associated matrices
$(A_1,A_2)$
 satisfies  one of the
following three conditions: 
\begin{enumerate}
 \item  \label{itemi}
the group $\langle A_1,A_2
\rangle$ is a finite group;
\item\label{itemii}
$p$ lies on one of the
coordinate axes of $\mathbb R^3$;
\item\label{itemiii}
the pair $(A_1,A_2)$ is conjugate
to $(A_1',A_2')$ where $A_1',A_2'$ are lower triangular and where the
diagonal elements of $A_1'$ and $A_2'$ are roots of unity.  This
includes the case where some word in $\langle A_1,A_2\rangle$ is
parabolic (so that the group $\langle A_1,A_2\rangle$ is not finite).
\end{enumerate}

Moreover $\mathcal P=\mathcal F$.
 \end{theorem}

It is easy to see that $p \in \mathcal F$ in each of the above three
cases.

For (\ref{itemi}) in the above Theorem  we recall the well known fact that any finite
subgroup of $\mathrm{SL}(2,\mathbb C)$ is cyclic, binary dihedral or a subgroup of one of:

\noindent {\it The binary tetrahedral group} 
$$\mathrm{BT}_{24}=\bigg\langle \begin{pmatrix} i&0\\0&-i\end {pmatrix},
\begin{pmatrix} 0&1\\-1&0\end {pmatrix},\frac 1 2 
\begin{pmatrix} 1+i&i-1\\1+i&1-i\end {pmatrix}\bigg\rangle.$$

Note that $BT_{24} \cong \mathrm{SL}(2,3)$.

\noindent {\it The binary octahedral group} 
$$\mathrm{BO}_{48}=\bigg\langle \begin{pmatrix} i&0\\0&-i\end {pmatrix},
\begin{pmatrix} 0&1\\-1&0\end {pmatrix},\frac 1 2 
\begin{pmatrix} 1+i&i-1\\1+i&1-i\end {pmatrix},  
\frac 1 {\sqrt{2}} \begin{pmatrix} 1+i&0\\0&1-i\end {pmatrix}\bigg\rangle.$$

\noindent {\it The binary icosahedral group} 
$$\mathrm{BI}_{120}=\bigg \langle 
\begin{pmatrix} \varepsilon^3&0\\0&\varepsilon\end {pmatrix},
\begin{pmatrix} 0&1\\-1&0\end {pmatrix},
\frac 1 {\sqrt{5}} \begin{pmatrix}  \varepsilon^4- \varepsilon& \varepsilon^2- \varepsilon^3\\ \varepsilon^2- \varepsilon^3& \varepsilon- \varepsilon^4\end {pmatrix}\bigg \rangle,$$ 
where $ \varepsilon$ is a primitive fifth root of unity.

\medskip 

From the classification of Theorem \ref {th1}  it follows that if $p \notin E_1$ and $p$
does not lie on an axis, then the associated group $\langle A_1,A_2\rangle$ is finite, and $p$ is in one of  $\mathrm{Aut}(F_2)$-orbits  
$O_1,\dots,O_5$ of sizes $40, 36, 72, 16, 40$ (respectively).

The fact that $\mathrm{Inn}(F_2)$ acts trivially on the triples of  traces $(x_1,x_2,x_{12})$ means that we are really considering an action of $\mathrm{Out}(F_2)=\mathrm{Aut}(F_2)/\mathrm{Inn}(F_2)$.

We should note that in \cite {Dub} the authors consider the action of the braid group $B_3$ (instead of $\mathrm{Aut}(F_2)$) on such triples.
We further note however that this action of $B_3$ is determined  by the elements $\sigma_1,\sigma_2 \in \mathrm{Aut}(F_2)$ (defined in the next section). It is a fact that the subgroup $\langle \sigma_1,\sigma_2\rangle \cong \mathrm{PSL}(2,\mathbb Z)\cong B_3/Z(B_3)$ has finite index ($8$) in  $\mathrm{Out}(F_2)$. Thus the problem of determining points of $\mathbb R^3$  with finite orbit for the action of  $\mathrm{Aut}(F_2)$ is the same as the problem of finding points with finite orbit for the action of the subgroup $\langle \sigma_1,\sigma_2\rangle$.

  We also note that the authors of \cite {Dub}  are only interested in the triples $(x_1,x_2,x_{12}) \in \mathbb R^3$ (with finite orbit) up to an equivalence where two such are equivalent if they differ by changing the sign of two of the entries. Thus they obtain orbits of sizes $10,9,18,4,10$ (respectively). 

We also note that in \cite {Dub} they are only concerned with points having finite orbits that are in the interior of $\mathcal T$. Results given in $\S 3$, together with what is proved in \cite {Dub}, easily give Theorem \ref {th1}. 
The last thing to note about \cite {Dub} is that they use the traces of the products $A_{i_1}A_{i_2}\dots A_{i_k}$, not the half-traces.

\medskip

In the next few sections we introduce the preliminary results for the $n=2,3$ cases, finally being able to state our main result for the case $n=3$ at the end of section $2$. This result is proved in $\S\S 3-7$. The result for general $n$ is proved  in $\S 8$.

\medskip
\noindent {\bf Acknowledgements} We would like to thank Anthony Manning for his careful reading  and comments on this material. All computations made in the preparation of this manuscript were accomplished using Magma \cite {Ma}.

\section{Preliminaries for $n=2$}\label{sec1}




 Let $F_2=\langle a_1,a_2 \rangle$ be a free
group of rank $2$ and let $ {\sigma}_i \in \mathrm{Aut}(F_2), i=1,2,$
be defined by
\begin{eqnarray*}
 {\sigma}_1(a_1)&=&a_1a_2,\;\;  {\sigma}_1(a_2)=a_2;\\
 {\sigma}_2(a_1)&=&a_1,\;\;\;\;
 {\sigma}_2(a_2)=a_1^{-1}a_2. 
\end{eqnarray*}
One
can show that $ {\sigma}_1, {\sigma}_2$ satisfy the
braid relation 
$ {\sigma}_1 {\sigma}_2 {\sigma}_1=
 {\sigma}_2 {\sigma}_1 {\sigma}_2$ 
and   that 
$( {\sigma}_1 {\sigma}_2)^3$ acts as an inner automorphism, so that the action of $ ( {\sigma}_1 {\sigma}_2)^3$ on the trace triples  is trivial.
Under the natural homomorphism 
$$\Phi=\Phi_2:\mathrm{Aut}(F_2)\rightarrow \mathrm{GL}(2,\mathbb{Z}),$$ 
we have 
$$\Phi( {\sigma}_1)  = \begin{pmatrix} 1&1\\0
  &1\end{pmatrix}, \quad 
\Phi( {\sigma}_2)   = \begin{pmatrix}
  1&0\\-1&1\end{pmatrix}.$$ 
  
Note
  \cite[Th 3.9]{MKS}
that any element of $\mathrm{Aut}(F_2)$ fixes the commutator
$a_1a_2a_1^{-1}a_2^{-1}$ up to conjugacy and inversion.

 Now suppose
that the $a_i,i=1,2,$ are represented by elements  $A_1,A_2 \in \mathrm{SL}(2,\mathbb{
C})$.  Define $x_1,x_2,x_{12}$ as in (\ref{eq0.2}).  Recall the  standard trace
identities for such $2 \times 2$ matrices:
\begin{eqnarray*}
\mathrm{trace}(A_1^{-1})&=&\mathrm{trace}(A_1),\hspace{20mm}
\mathrm{trace}(I_2)=2, \\
\mathrm{trace}(A_1A_2)&=&\mathrm{trace}(A_2A_1)=\mathrm{trace}(A_1)\mathrm{trace}(A_2)
-\mathrm{trace}(A_1A_2^{-1}).  
\end{eqnarray*}
Using these we obtain the  induced action of $\sigma_1, \sigma_2$
on $\mathbb{R}^3$
 given by (\ref{eq0.1}).

We will write this action of $\alpha \in {\rm Aut}(F_2)$ on
$(x,y,z)^T \in \mathbb{ R}^3$ on the right: $(x,y,z)^T\alpha$; this
action is also the corresponding action by Nielsen transformations
\cite{MKS,RB2}.  In \cite{hm1,hm2} we studied  the family  of trace maps
$\sigma_1^n\sigma_2^n$ and determined all of their curves of fixed
points and some of their period $2$ curves.

One can check that for all $A_1,A_2 \in \mathrm{SL}(2,\mathbb C)$ we have $\mathrm{trace}(A_1A_2A_1^{-1}A_2^{-1})/2=2E(x_1,x_2,x_{12})-1$.
Thus, from the above trace identities and the fact that any element of $\mathrm{Aut}(F_2)$ fixes the
commutator $a_1a_2a_1^{-1}a_2^{-1}$ up to conjugacy and inversion, it
follows that the action of  $\mathrm{Aut}(F_2)$ fixes the function
$E=E(x_1,x_2,x_{12})$ of (\ref{eq0.3}). 
 Thus each level set $E_t$ is invariant under the
action.  The level set $E_1$ is distinguished and has been drawn by
many authors \cite{Go1,RB1}.  There are four points,
$$V=\{(1,1,1)^T,(-1,-1,1)^T,(-1,1,-1)^T,(1,-1,-1)^T\},$$ 
in $E_1$  which
are the only singular points of $E_1$.  In fact they are the only
singular points of any $E_t$.  Further, the six line segments 
joining these points are contained in $E_1$ and there is a unique
component of $E_1 \setminus V$ whose closure is compact.  In fact
this closure is a topological $2$-sphere that separates $\mathbb{ R}^3$
into two components, the closure of one of these components is a
$3$-ball $\mathcal T$ that we call a ``curvilinear tetrahedron", for
whose boundary we gave a parametrisation in (\ref{eq0.4}).  One can check
that $\mathcal T \subseteq [-1,1]^3$ and that $\mathcal T \cap
\partial [-1,1]^3$ is the above mentioned set of six line segments.
Further the closures of the other (four) components of $E_1 \setminus V$ are determined by the point of $V$ which they contain; we will call these components {\it cones}.

Let $\mathbb{ T}^2=\mathbb{ R}^2/\mathbb{ Z}^2$ denote the $2$-torus.  Then we have  the
map
$$\Pi_2:\mathbb{ T}^2 \to \partial \mathcal T,\quad (\theta_1,\theta_2)^T \mapsto
(\cos(2\pi \theta_1),\cos(2\pi
\theta_2),\cos(2\pi(\theta_1+\theta_2)))^T.$$ Note that
$\Pi_2(\theta_1,\theta_2)^T=\Pi_2(-(\theta_1,\theta_2)^T)$.  Then the
map $\Pi_2$ is a branched double cover, branched over the four
points of $V.$

The action of  $\mathrm{Aut}(F_2)$ on $\partial \mathcal T$ actually
comes from the action of $\mathrm{GL}(2,\mathbb{Z})$ on $\mathbb{ T}^2$, the
action being determined by the 
homomorphism $\Phi_2$. 
 For $\alpha \in \mathrm{Aut}(F_2)$ and 
$\theta=\begin{pmatrix} 
\theta_1\\ \theta_2\end{pmatrix} \in  \mathbb{T}^2$ the maps $\Pi_2, \Phi_2$
are related as follows (see \cite [p. 1170] {hm1}):
\begin{equation}\label{eq1.1}
(\Pi_2\theta)\alpha=\Pi_2(\Phi_2(\alpha)(\theta)).  
\end{equation}

The set $\mathcal F=\mathcal F_2$ includes the points $V$.  If we ignore the points
of $V$ for the moment, then, as pointed out in \cite[p. 839]{RB1}, a
consequence of the implicit function theorem is that, for any point
$p \in \mathcal P \cap
\partial \mathcal T$ and  $\sigma\in  \mathrm{Aut}(F_2)$ with
$\Phi_2(\sigma)$ hyperbolic, there is a curve of
fixed points of $\sigma^N$ through $p$ for some $N=N(p,\sigma)$; see
also \cite[\S5]{Br1}.

To every point $(x_1,x_2,x_{12})^T \in \mathbb{ R}^3$ there is a pair $(A_1,A_2) \in
\mathrm{SL}(2,\mathbb{ C})^2$ such that trace$(A_1)=X=2x_1,$ trace$(A_2)=Y=2x_2$ and
trace$(A_1A_2)=Z=2x_{12}$.  The pair $(A_1,A_2)$ is determined up to conjugacy
by the point $(x_1,x_2,x_{12})^T\notin E_1$ and one possible choice for $A_1, A_2$ is
\begin{equation*}
A_1=\begin{pmatrix} \frac  {X+\sqrt{X^2-4}} {2}&0\\0&\frac {2}
{X+\sqrt{X^2-4}}\end{pmatrix},
\end{equation*}
\begin{equation}\label{eq1.2}
A_2=
\begin{pmatrix} \frac {XZ+Z\sqrt{X^2-4}-2Y} {X^2-4+X\sqrt{X^2-4}}&\frac
{X^2+Y^2+Z^2-XYZ-4}{4-X^2}\\
1&\frac
{-XZ-Z\sqrt{X^2-4}-2Y+X^2Y+XY\sqrt{X^2-4}}{X^2-4+X\sqrt{X^2-4}}\end{pmatrix}.  
\end{equation}
We note that  $(x_1,x_2,x_{12}) \in E_1$ if and
only if the $(1,2)$ entry of the above $A_2$ is zero, so that $A_1,A_2$
are both lower triangular matrices if $(x_1,x_2,x_{12}) \in E_1$.

It follows from the proof of Theorem \ref{th1} in \cite {Dub} that there are only
five finite  $\mathrm{Aut}(F_2)$ orbits which do not lie in $E_1$ or on the
axes.  We denote them by $O_1,\dots,O_5$ and they contain (respectively)
$40,36,72,16,40$ points.  They lie on levels 
$$0.904508\ldots,\;\;
3/4,\;\; 3/4,\;\; 1/2,\;\; 0.3454915\ldots$$ 
and each has
associated group $\langle A_1,A_2\rangle$ a finite group of order
$120,48,120,24,120$ (respectively), the group of order $120$ being
$BI_{120}$.  (See \cite[Th 29.6]{Hup} for the possible
finite subgroups of $SU(2,\mathbb{C})$.)  Thus there are four levels
that contain the five sets $O_i, i=1,\dots ,5$.

In order to give the reader an idea of what the points of each $O_i$  look like
we draw the following diagrams.  For every point $p=(x,y,z)^T \in
O_i$, we take the 
simple closed curves containing $p$ which are parallel to the
various coordinate planes and which lie on the same level as $p$;
there will be three such curves through every $p\in O_i$.  We
then radially project these curves onto the unit sphere and then
stereographically project  onto the $xy$-plane.  These curves are
 drawn in Figures $1,2$.  The points of $O_i$ in the Figures  are exactly
those points which are triple points for these curves.  We further
note that each such simple closed curve contains many of the
points of $O_i$, all such points in a single curve being in one or
more orbits of a conjugate of $\sigma_1$.  


\begin{figure}
     \centering
     \subfigure[$O_1$]{
          \includegraphics[width=.45\textwidth]{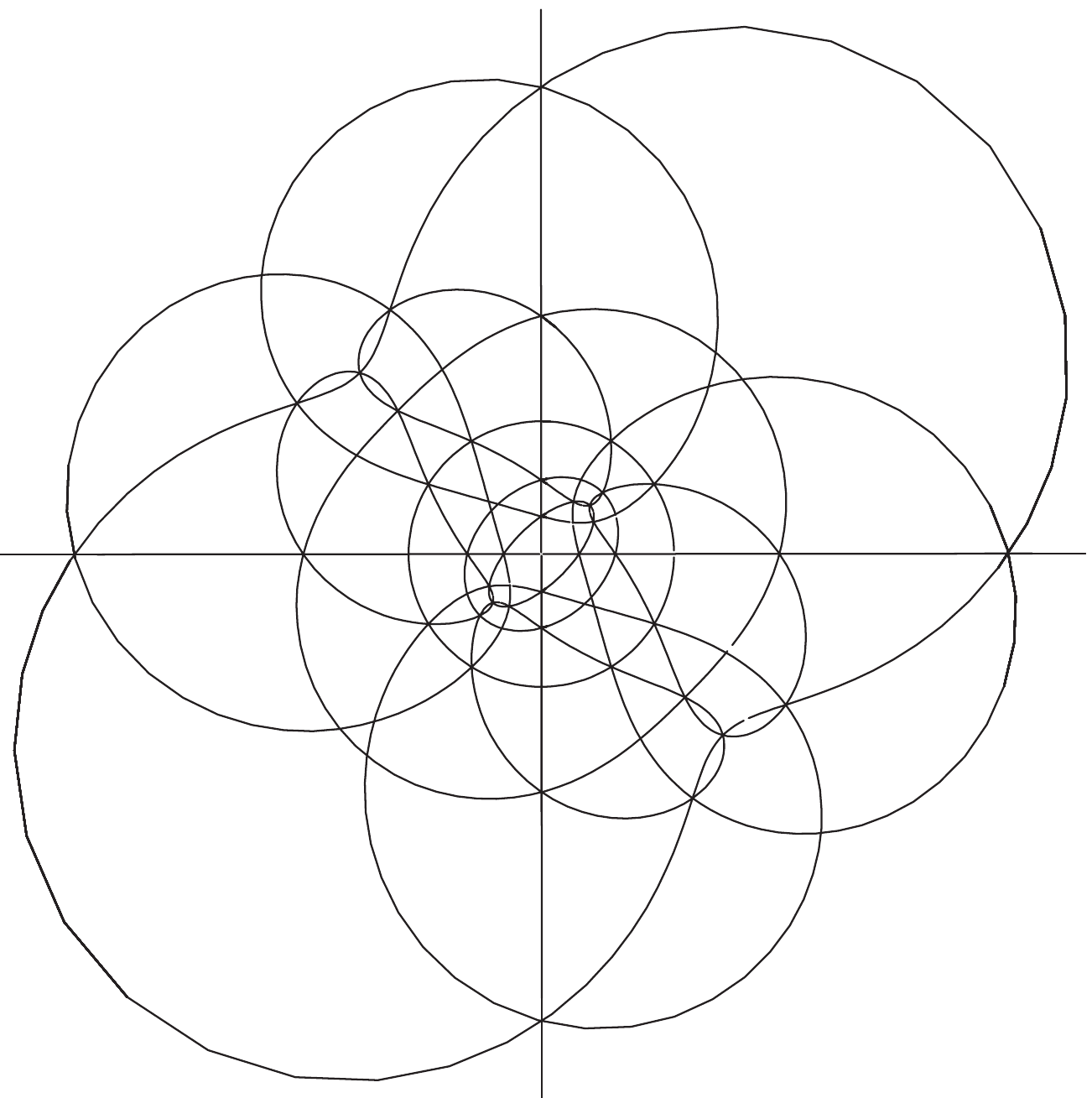}}
     \hspace{.3in}
     \subfigure[$O_2$]{
          \includegraphics[width=.45\textwidth]{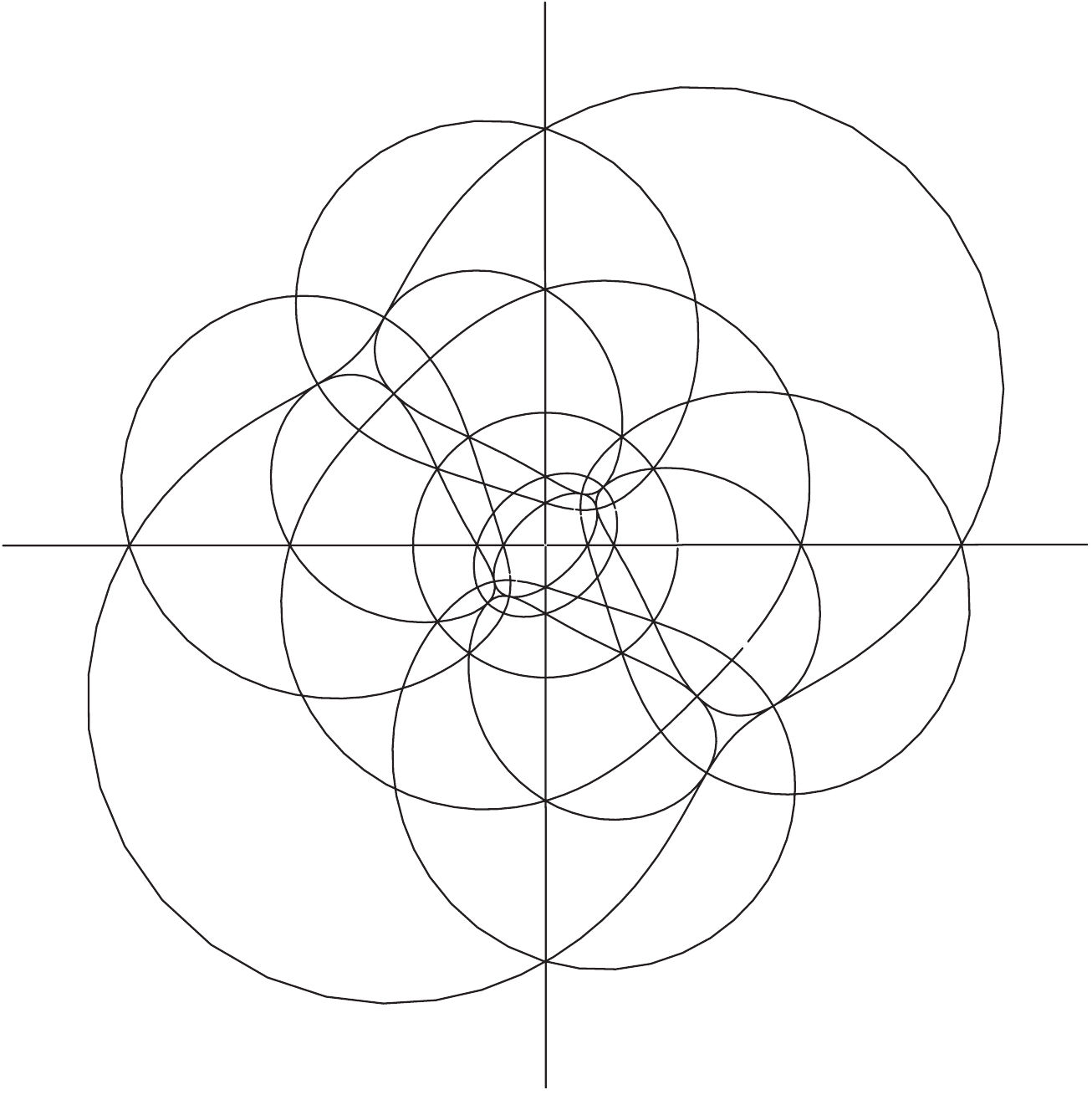}}
\caption{}
\end{figure}

\begin{figure}[b]
     \centering
     \subfigure[$O_3$]{
          \includegraphics[width=.65\textwidth]{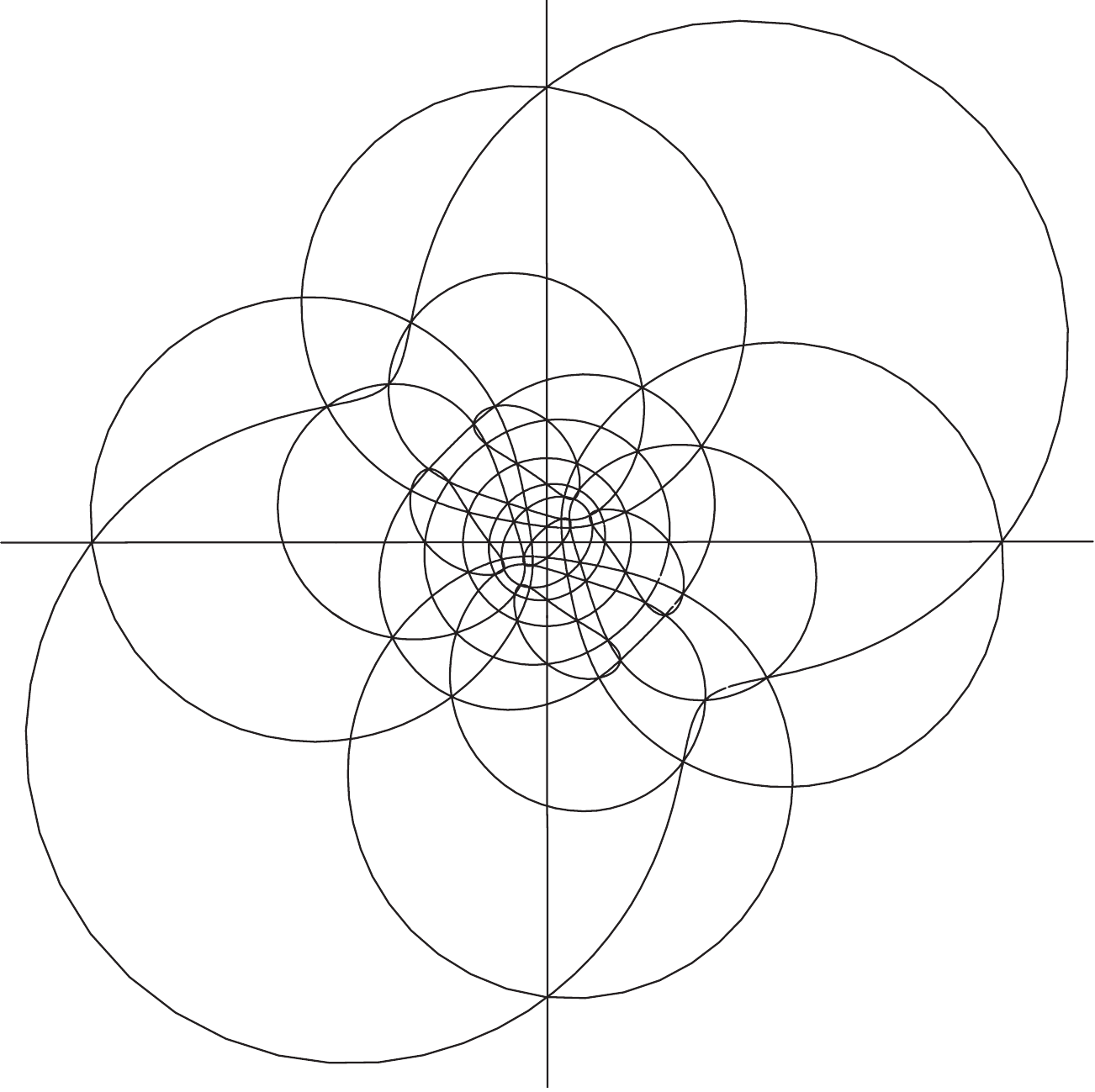}}
     \hspace{.1in}
     \subfigure[$O_4$]{
           \includegraphics[width=.45\textwidth]
                {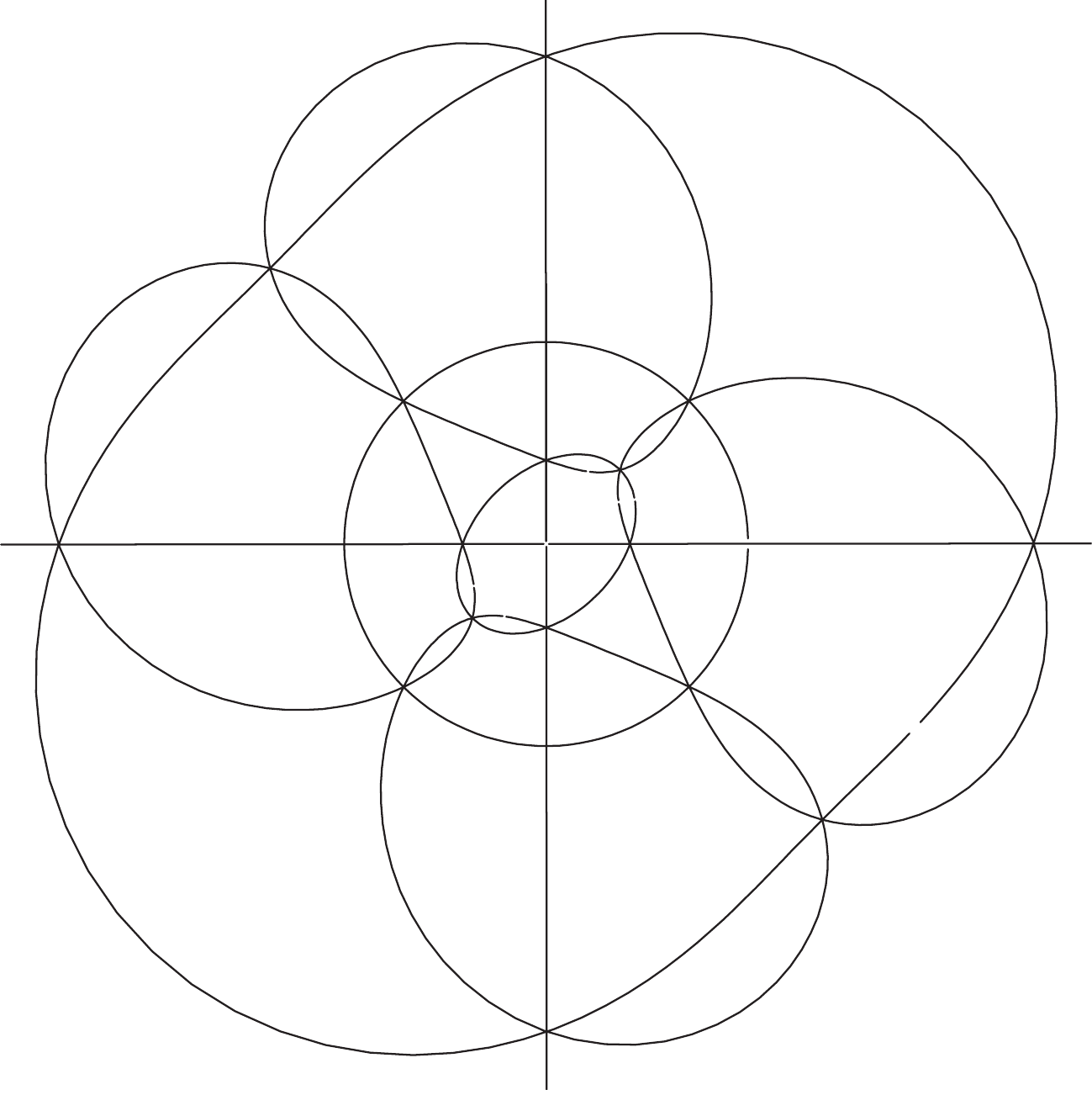}}
     \subfigure[$O_5$]{
          \includegraphics[width=.45\textwidth]{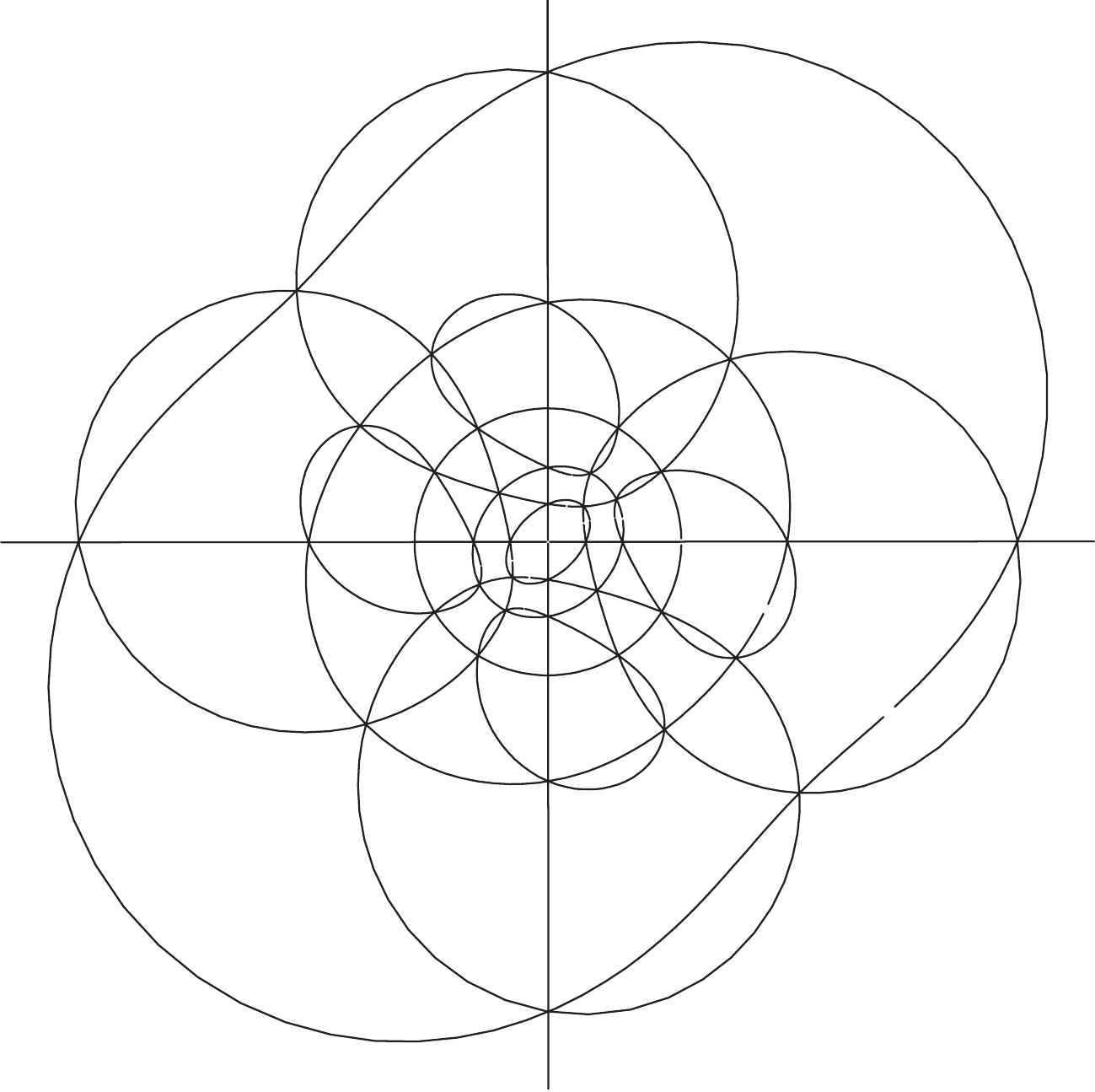}}
\caption{}
\end{figure}

From the above one easily deduces:
\begin{corollary} \label{corvals}
The half-trace values for the points in the  orbits $O_i,1\le i\le 5,$ are
$$0,\pm 1/2, \pm 1/\sqrt {2}, (\pm 1 \pm \sqrt {5})/4.\qed $$
\end{corollary}


We now list further  properties of the orbits $\mathcal O_1,\dots, \mathcal O_5$.   
\medskip

{\bf 72} There is a single $\mathrm{Aut}(F_2)$  orbit $O_3$ of size
$72$.  This orbit is on level $3/4$.  The permutation action of
$\mathrm{Aut}(F_2)$ on $O_3$ 
gives a group of order $2^{34}\cdot 3^5\cdot 5\cdot 7$ which is an
imprimitive permutation group: the $72$ points can be partitioned
into  $9$ blocks of size $8$.  The matrix group $\langle
A_1,A_2\rangle$ associated to any point of this orbit is the binary 
icosahedral group  $\mathrm{BI}_{120}$.
\medskip

{\bf 40} There are two $\mathrm{Aut}(F_2)$   orbits $O_1, O_5$ of size
$40$.  These 
are on levels $(3+\zeta_5^2+\zeta_5^3)/4=0.34549\dots $ and
$(2-\zeta_5^2-\zeta_5^3)/4=0.904508\dots$ (respectively), where $\zeta_5=\exp{2\pi i/5}$.  The
action of $\mathrm{Aut}(F_2)$ on either of these orbits of size $40$ gives an
imprimitive group with ten blocks of size $4$, the action on the
ten blocks being that of $S_{10}$.  (Points in a block of size $4$ are
obtained from each other by changing the sign of two coordinates.)  
The elements of the two sets
$O_1, O_5$ are defined over the cyclotomic field $\mathbb{ Q}(\zeta_5)$
and are interchanged by the Galois automorphism $\alpha:\mathbb{
Q}(\zeta_5) \to \mathbb{ Q}(\zeta_5), \alpha(\zeta_5)=\zeta_5^2$.  The
matrix group $\langle A_1,A_2\rangle$ associated to any point of these
orbits is  also $\mathrm{BI}_{120}$.
\medskip

{\bf 36} There is a single $\mathrm{Aut}(F_2)$   orbit $O_2$ of size $36$.  This
orbit is on level $3/4$.  This case has permutation group of order
$2^{16}\cdot 3^4$ and is imprimitive with blocks of sizes
$2,4,12$.   The associated group $\langle A_1,A_2\rangle$ is $\mathrm{BO}_{48}$.
\medskip

{\bf 16} There is a single $\mathrm{Aut}(F_2)$   orbit $O_4$ of size $16$.  The associated 
matrix group $\langle A_1,A_2\rangle$ is $\mathrm{BT}_{24}$ and is isomorphic
to $\mathrm{SL}(2, \mathbb{ F}_3)$.  The permutation group has order $2^{12}\cdot 3^2$
and has blocks of size $4$.  This orbit is on level $1/2$.  
\medskip


Any point  $p \in \partial\mathcal T$ with finite $\mathrm{Aut}(F_2)$-orbit has the form
$\Pi(r/n,s/n)^T$ for some $(r/n,s/n)^T$ in $\mathbb{ Q}^2$.  Then the
action of $\mathrm{Aut}(F_2)$ on $p$ is determined by the action of
$\mathrm{SL}(2,\mathbb 
Z)$ on $(r/n,s/n)^T$, this action being given by $\mathrm{SL}(2,\mathbb{ Z}_n)$.
Thus the kernel of the permutation action  on all such points is a congruence
subgroup.


\vspace{4mm}

We may also discuss  the 
$\mathrm{Aut}(F_2)$-action on the finite orbits 
$O_1,O_2,O_3,O_4,O_5$ in the following way.   

If 
$|y|\leq 1$ 
then the restriction of the map 
$\sigma_1: (x,y,z)\mapsto (z,y,2yz-x)$ 
to the affine plane 
$y=\cos(2\pi\rho_y)$ 
is topologically conjugate to a rotation through the angle 
$2\pi\rho_y$, 
because the matrices 
$$ \left( \begin{array}{cc}
0&1\\ -1&2y \end{array} \right), \;\;
\left( \begin{array}{cc}
\cos(2\pi\rho_y)&-\sin(2\pi\rho_y) \\ 
\sin(2\pi\rho_y)&\cos(2\pi\rho_y) \end{array} \right) 
$$
are similar.   
Therefore the ellipse 
$x^2+z^2-2xyz=t-y^2$, 
in which the plane 
$y=\cos(2\pi\rho_y)$ 
meets 
$E_t,$ 
is mapped to itself by 
$\sigma_1$ 
with rotation number 
$\rho_y$.   
Thus points of the finite orbits 
$O_1,O_2,O_3,O_4,O_5$ 
must have 
$y$
(and similarly 
$x,z$) 
of the form 
$\cos(2\pi p/q)$.   
The following table lists the number of points of each 
$O_i$ 
in each such ellipse (including the case 
$\rho_0 = \frac{1}{4}$, 
which is a circle in the coordinate plane 
$y=0$).   
It is these points that are represented in Figures 1 and 2 as the
triple points of intersection of ellipses (after stereographic
projection).   
Our 
$\sigma_1$ 
rotates parallel slices by related amounts, in contrast to the Rubik
cube in which these slices can be rotated independently!  We indicate this in the 
following table.

$$
\begin{array}{|c|c|c|c|c|c|c|c|c|c|c|c|c|c|c|}
\hline
\rho_y&\frac{2}{5} &\frac{3}{8} & \frac{1}{3} & \frac{3}{10} & 
\frac{1}{4}&  \frac{1}{5} & \frac{1}{6} & \frac{1}{8} & \frac{1}{10} 
& \mbox{Total} & t & \mbox{Typical points} \\
\hline
O_1& & &6 &10 &8 &10 &6 & & &40 & \frac{5+\sqrt5}{8} & 
(\frac{1}{2},\frac{\sqrt5+1}{4},\frac{\sqrt5+1}{4})^T \\
&&&&&&&&&&&&
(\frac{\sqrt5+1}{4},\frac{\sqrt5+1}{4},\frac{\sqrt5+1}{4})^T  \\
\hline
O_2& &8 &6 & &8 & &6 &8 & &36 &\frac{3}{4} & 
(\frac{1}{2}\sqrt2,\frac{1}{2},\frac{1}{2}\sqrt2)^T \\
&&&&&&&&&&&&
(0,\frac{1}{2},\frac{1}{2}\sqrt2)^T  \\
\hline
O_3& 10& &12 &10 &8 &10 &12 & &10 &72 & \frac{3}{4} & 
(\frac{1}{2},\frac{\sqrt5-1}{4},\frac{\sqrt5+1}{4})^T \\
&&&&&&&&&&&&
(0,\frac{\sqrt5-1}{4},\frac{\sqrt5 +1}{4})^T  \\
\hline
O_4& & &6 & &4 & &6 & & &16 &\frac{1}{2} & 
(\frac{1}{2},\frac{1}{2},\frac{1}{2})^T \\
&&&&&&&&&&&&
(\frac{1}{2},0,\frac{1}{2})^T  \\
\hline
O_5& & &6 &10 &8 &10 &6 & & &40 & \frac{5-\sqrt5}{8} & 
(\frac{1}{2},\frac{1-\sqrt5}{4},\frac{1-\sqrt5}{4})^T \\
&&&&&&&&&&&&
(\frac{1-\sqrt5}{4},\frac{1-\sqrt5}{4},\frac{1-\sqrt5}{4})^T  \\
\hline
\end{array}
$$
\medskip

\section{Preliminaries for $n=3$}\label{sec2}

 As we have seen in the above, in the $n=2$ case there is an invariant of the action 
$E(x,y,z)=x^2+y^2+z^2-2xyz$ that comes from trace$(a_1^{-1}a_2^{-1}a_1a_2)$. This `Fricke character' however is also determined by the fact that for real $\theta_1,\theta_2$ 
the values $x=\cos \theta_1, y=\cos \theta_2, z=\cos(\theta_1+\theta_2)$ satisfy the relation 
$E(  x,y,z ) =1$. 
In fact the relation $E(x,y,z)-1=0$ generates the ideal of polynomials $p \in \mathbb Q[x,y,z]$ such that
$p(\cos \theta_1, \cos \theta_2, \cos(\theta_1+\theta_2))=0$ for all  $\theta_1,\theta_2\in \mathbb R$.

Further, the relevance of   $x=\cos \theta_1, y=\cos \theta_2, z=\cos(\theta_1+\theta_2)$ is that they are the half-traces of the  matrices
 $A_i=\begin{pmatrix} \cos \theta_i+i\sin \theta_i&0\\
0&\cos \theta_i-i\sin \theta_i\end {pmatrix}, i=1,2$ and
 $A_1A_2=\begin{pmatrix} \cos (\theta_1+\theta_2)+i\sin (\theta_1+\theta_2)&0\\
0&\cos (\theta_1+\theta_2) -i\sin (\theta_1+\theta_2)\end {pmatrix}$. We note that the triple   $A_1,A_2,A_1A_2$ pair-wise  commute. Thus they do not give generic values for $x,y,z$. 
However they give special values of $x,y,z$ that satisfy $E=1$.

In the case $n=3$   that we are   considering next these two different ways of understanding $E$ give two different algebraic objects, one of which will be an invariant for the action of $\mathrm{Aut}(F_3)$, while the other will turn out to be an ideal.

To find the ideal just referred to we consider three diagonal commuting matrices:
If $A_i=\begin{pmatrix} \cos \theta_i+i\sin \theta_i&0\\
0&\cos \theta_i-i\sin \theta_i\end {pmatrix}$, for $i=1,2,3$,  then  $ x_i=\cos\theta_i,  x_{ij}=\cos (\theta_i+\theta_j), x_{ijk}=\cos (\theta_i+\theta_j+\theta_k)$.
We wish to find the ideal of relations satisfied by these $x_i,x_{ij},x_{ijk}$.
Then in this situation $ x_i=\cos \theta_i,i=1,2,3$,  and we let  $y_i=\sin \theta_i,i=1,2,3$,
 $ x_{ij}=\cos (\theta_i+\theta_j),i,j=1,2,3$,  and similarly $y_{ij}=\sin (\theta_i+\theta_j),i,j=1,2,3$,
 $ x_{ijk}=\cos (\theta_i+\theta_j+\theta_k),i=1,2,3$, and $y_{ijk}=\sin (\theta_i+\theta_j+\theta_k),i,j,k=1,2,3$.

 Let $\mathcal I$ be the ideal of $\mathbb Q[x_1,x_2,\dots,x_{123},y_1,y_2,\dots,y_{123}]$ consisting of all polynomials $f(  x_1,x_2,\dots,x_{123},y_1,y_2,\dots,y_{123})$ such that
 $$f(\cos \theta_1,\cos \theta_2,\dots, \cos (\theta_1+\theta_2+\theta_3),
    \sin \theta_1, \sin \theta_2, \dots  , \sin (\theta_1+\theta_2+\theta_3))=0.$$

Then  using basic trigonometric  identities we  see that $\mathcal I$  contains  the polynomials
\begin{align*}&
 x_1^2+y_1^2-1,\quad 
 x_2^2+y_2^2-1,\\&
 x_3^2+y_3^2-1,\quad 
 x_{12}^2+y_{12}^2-1,\\ &
 x_{13}^2+y_{13}^2-1,\quad 
 x_{23}^2+y_{23}^2-1,\quad  x_{123}^2+y_{123}^2-1,\\&
 x_{12}- x_1  x_2+y_1 y_2,\quad 
 x_{23}- x_2  x_3+y_2 y_3,\\ &
 x_{13}- x_1  x_3+y_1 y_3,\quad 
y_{12}-y_1  x_2- x_1 y_2,\\ &
y_{23}-y_2  x_3- x_2 y_3,\quad 
y_{13}-y_1  x_3- x_1 y_3,\\&
 x_{123}- x_1  x_{23}+y_1 y_{23},\quad 
y_{123}-y_1  x_{23}- x_1 y_{23}.
\end{align*}

Finding a Gr\"obner basis for $\mathcal I$ (using \cite {Ma}), and then doing an elimination we obtain the following relations among the $ x_i, x_{ij}, x_{ijk}$:
\begin{align*}&
 x_3^2 - 2  x_3  x_{12}  x_{{12}3} +  x_{12}^2 +  x_{{12}3}^2 - 1,\\&
     x_2^2 - 2  x_2  x_{13}  x_{{12}3} +  x_{13}^2 +  x_{{12}3}^2 - 1,\\&
     x_2  x_3  x_{{12}3}^2 - \frac {1} {2}   x_2  x_3 - \frac {1} {2}   x_2  x_{12}  x_{{12}3} - \frac {1} {2}   x_3  x_{13}  x_{{12}3} + \frac {1} {2}   x_{12}  x_{13} - \frac {1} {2}   x_{23}  x_{{12}3}^2 + \frac {1} {2}   x_{23},\\&
     x_1 + 2  x_2  x_3  x_{{12}3} -  x_2  x_{12} -  x_3  x_{13} -  x_{23}  x_{{12}3},\end{align*} 
     \begin{align*}&
     x_2  x_{12}  x_{23} - 2  x_2  x_{13}  x_{{12}3}^2 +  x_2  x_{13} + 2  x_3  x_{12}  x_{{12}3}^2 -  x_3  x_{12} -  x_3  x_{13}  x_{23} -  x_{12}^2  x_{{12}3} +  x_{13}^2  x_{{12}3},\\&
     x_2  x_3  x_{12} - 2  x_2  x_{{12}3}^3 +  x_2  x_{{12}3} -  x_3  x_{23}  x_{{12}3} -  x_{12}^2  x_{13} + 2  x_{12}  x_{23}  x_{{12}3}^2 -  x_{12}  x_{23} +  x_{13}  x_{{12}3}^2,\\&
     x_2  x_3  x_{23} -  x_2  x_{13}  x_{{12}3} -  x_3  x_{12}  x_{{12}3} + \frac {1} {2}   x_{12}^2 + \frac {1} {2}   x_{13}^2 - \frac {1} {2}   x_{23}^2 +  x_{{12}3}^2 - \frac {1} {2}. 
\end{align*} 

Let $R=\mathbb Q[ x_1, x_2, x_3, x_{12}, x_{13}, x_{23},  x_{123}]$ and let $\mathcal X \subset R$ denote the ideal generated by these polynomials.
One finds that $\mathcal X$ has dimension $3$ with $ x_{13}, x_{23}, x_{123}$ being algebraically independent. The ideal $\mathcal X$ is one of the analogues of the Fricke character $E$ from the $n=2$ situation. 


We define the following automorphisms of $F_3$:
\begin{align*} &
U: [a_1,a_2,a_3] \mapsto [a_1a_2,a_2,a_3];\\&
Q: [a_1,a_2,a_3] \mapsto [a_2,a_3,a_1];\\&
S: [a_1,a_2,a_3] \mapsto [a_1^{-1},a_2,a_3]\\&
P: [a_1,a_2,a_3] \mapsto [a_2,a_1,a_3].
\end{align*}
It is well-known \cite [p. 164]  {MKS} that Aut$(F_3)$ is generated by $U,P,S,Q$.  The action (on the left) of   ${\rm Aut}(F_3)$ on $R$ is determined by
the actions of the generators  $U,P,S,Q$:
\begin{align*}
&U([   x_{1},  x_{2},  x_{3},  x_{12},  x_{13},  x_{23},  x_{123} ])=[ x_{12}, x_{2}, x_{3},2  x_{2}  x_{12}- x_{1}, x_{123}, x_{23},2  x_{2}  x_{123}- x_{13}];\\&
Q([   x_{1},  x_{2},  x_{3},  x_{12},  x_{13},  x_{23},  x_{123} ])=[ x_{2}, x_{3}, x_{1}, x_{23}, x_{12}, x_{13}, x_{123}];\\&
S([   x_{1},  x_{2},  x_{3},  x_{12},  x_{13},  x_{23},  x_{123} ])\\&\qquad 
=[x_1,
    x_2,
    x_3,
    2  x_1  x_2 -  x_{12},
    2  x_1  x_3 -  x_{13},
     x_{23},
    2  x_1  x_{23} -  x_{123}
];\\&
P([    x_{1},   x_{2},   x_{3},   x_{12},   x_{13},   x_{23},   x_{123} ])\\&\qquad =[
     x_2,
     x_1,
     x_3,
     x_{12},
     x_{23},
     x_{13},
    -4  x_1  x_2  x_3 + 2  x_1  x_{23} + 2  x_2  x_{13} + 2  x_3  x_{12} -  x_{123}
].
\end{align*}
This also gives the actions of $U,P,S,Q$ (on the right) on elements of $\mathbb R^7$. We now show that this really does determine an action of
Aut$(F_3)$ on $R$ and on $\mathbb R^7$:

\begin{lemma} \label{lemact}
The above action of the generators $U,P,S,Q$ determines a homomorphism ${\rm Aut}(F_3) \to {\rm Aut}(R)$ and an anti-homomorphism 
 ${\rm Aut}(F_3) \to {\rm Homeo}(\mathbb R^7),$ such that for all $\sigma \in {\rm Aut}(F_3)$ we have
 $$\sigma(x_1,x_2,x_3,x_{12},x_{13},x_{23},x_{123})=(x_1,x_2,x_3,x_{12},x_{13},x_{23},x_{123})\sigma.$$

Further, the ideal $\mathcal X$ is invariant under this action.
\end{lemma}
\noindent {\it Proof}  A presentation for Aut$(F_3)$ is given in \cite [p. 164]  {MKS}:
\begin{align*}\langle U,Q,P,S|&
P^2,
Q^3,
S^2,
(Q P)^2,
S Q P=Q P S,
S Q ^{-1} P Q=Q ^{-1} P Q S,\\&
S Q ^{-1} S Q=Q ^{-1} S Q S,
(U,Q P Q ^{-1} P Q),
(U,Q ^{-2} S Q^2),
(P S P U)^2,\\&
P U P S U S P S=U,
(P Q ^{-1} U Q)^2 U Q ^{-1}=U Q ^{-1} U,
U S U S=S U S U, \\&
(U,P Q ^{-1} S U S Q P),
(U,P Q ^{-1} P Q P U P Q ^{-1} P Q P)\rangle.
\end{align*}
(The presentation from  \cite [p. 164]  {MKS}  seems to have two further relations, however these relations only apply to the situation $n>3.$)
 Thus to check that we do obtain a homomorphism
Aut$(F_3) \to {\rm Aut}(R)$ we just need to show that each relation in the presentation acts trivially on $R$. This is straight-forward. The fact that  $\sigma(x_1,x_2,x_3,x_{12},x_{13},$ $ x_{23},x_{123})=(x_1,x_2,x_3,x_{12},x_{13},x_{23},x_{123})\sigma$ is proved by induction on the length of $\sigma$ as a word in the generators $U,Q,S,P$.

Now  to check that the ideal $\mathcal X$ is invariant under this action we just show that for each element $x$ of a basis  for $\mathcal X$ we have $U(x),P(x),S(x),Q(x) \in \mathcal X$.
This is also straight-forward.\qed\medskip


We will need the following result using an element discovered by Horowitz \cite {ho}; this gives us an element that is fixed by Aut$(F_3)$ (thus this element also corresponds to $E$ in the $n=2$ case):
\begin{lemma} \label{lemfix}
The element
\begin{align*}
&F= x_1^2 
+  x_2^2 
+  x_3^2
+  x_{12}^2
+ x_{13}^2 +  x_{23}^2 +  x_{123}^2 +
4  x_1  x_2  x_3  x_{123} \\&- 2  x_1  x_2  x_{12} - 2  x_1  x_3  x_{13} - 2  x_1  x_{23}  x_{123 }- 
    2  x_2  x_3  x_{23} - 2  x_2  x_{13}  x_{123 } - 2  x_3  x_{12}  x_{123}\\&  + 2  x_{12}  x_{13}  x_{23} -1,
    \end{align*}
    is in $\mathcal X$ and 
    is fixed by each   $\alpha \in 
   {\rm Aut}(F_3).$
    \end{lemma}
    \noindent {\it Proof} For the first statement we use a  Gr\"obner basis  for $\mathcal X$ to show that $F \in \mathcal X$.
    For the last statement we just   show that $U(F)=P(F)=S(F)=Q(F)=F$.\qed\medskip
    
    Let $V(\mathcal X)$ denote the (real)  variety corresponding to  $\mathcal X$, and let $V(F)$ be that for $F$.
 
 We note that by the results of  \cite {ho}, the polynomial $F$ generates the principal ideal of all trace relations for three matrices in $\mathrm{SL}(2,\mathbb C)$.  Thus for all
 trace septuples $p=(x_1,x_2,x_3,x_{12},x_{13},x_{23},x_{123})$ determined by a triple $A_1,A_2,A_3 \in {\rm SL}(2,\mathbb C)$ we have $F(p)=0$ i.e. $p \in V(F)$. {\it Thus in what follows we will always assume that $p \in V(F)$.}

    \medskip

We have already noted that $\mathcal X$ has dimension $3$. This next result gives a parametrization of a compact part of $V(\mathcal X)$ of dimension $3$ that corresponds to the curvilinear tetrahedron $\mathcal T$ in the $n=2$ case.  For $n \in \mathbb N$ we let $\mathbb T^n$ denote the $n$-torus $(S^1)^n$. 
\begin{lemma}\label{lemparam}
Define $\Pi_3:\mathbb R^3 \to V(\mathcal X),$ by
\begin{align*}
& (t_1,t_2,t_3) \mapsto (\cos 2\pi t_1,\cos 2\pi t_2,\cos 2\pi t_3,\cos 2\pi (t_1+t_2),\cos 2\pi (t_1+t_3),\\&\qquad \cos 2\pi (t_2+t_3),\cos 2\pi (t_1+t_2+t_3)).
\end{align*}
Then $\Pi_3$ factors through
$$\mathbb R^3 \to \mathbb T^3=(\mathbb R/\mathbb Z)^3 \to \mathbb T^3/(\pm 1) \to V(\mathcal X),$$ and has image homeomorphic to  $\mathbb T^3/(\pm 1) $. 
\end{lemma} 
\noindent {\it Proof}  To check that the image of $\Pi_3$ is in $V(\mathcal X)$ we just take each element $x=x(x_1,x_2,x_3,x_{12},x_{13},x_{23},x_{123})$ of a basis for $\mathcal X$ and show that
$x(\Pi_3(t_1,t_2,t_3))=0$. This is straightforward. The rest follows.\qed.\medskip

Next we determine some   copies of  $\mathbb R^3$ inside $\mathbb R^7$.
If we choose matrices
$$A_1=\begin{pmatrix} \lambda&0\\0& {\lambda^{-1}}\end{pmatrix},\quad A_2=\begin{pmatrix} \mu&0\\0& {\mu^{-1}}\end{pmatrix},\quad 
A_3=\begin{pmatrix} 0&1\\-1&0\end{pmatrix},$$
and we let $ y_i={\rm trace}(A_i)/2,  y_{ij}={\rm trace}(A_iA_j)/2, y_{ijk}={\rm trace}(A_iA_jA_k)/2$, then 
$$ y_3= y_{13}= y_{23}= y_{123}=0.$$ We say that $( y_1,y_2,y_3,y_{12},y_{13},y_{23}, y_{123})$  {\it has zeros in positions} $3,5,6,7$.
We have:
\begin{lemma}\label{lems7}
If $\alpha \in {\rm Aut}(F_3)$ and $y_1,y_2,y_3,y_{12},y_{13},y_{23}, y_{123}$ are as defined above, then $(y_1,y_2,y_3,y_{12},y_{13},y_{23}, y_{123})\alpha$ has zeros in one of the sets
\begin{align*}
  \{ 1, 2, 3, 7 \},
     \{ 1, 2, 5, 6  \},
     \{ 1, 3, 4, 6  \},
     \{ 1, 4, 5, 7  \},
     \{ 2, 3, 4, 5  \},
     \{ 2, 4, 6, 7  \},
     \{ 3, 5, 6, 7 \}.\end{align*}
 This gives a homomorphism $\Sigma_7:{\rm Aut}(F_3) \to S_7$, where
 \begin{align*}&
 \Sigma_7(U)=(1, 5)(2, 6); \quad \Sigma_7(Q)=(2, 3, 5)(4, 7, 6);\\
 & \Sigma_7(S)=1; \quad \Sigma_7(P)=(3, 5)(4, 6).
 \end{align*}
 Lastly,
 $\langle \Sigma_7(U), \Sigma_7(Q), \Sigma_7(S),  \Sigma_7(P)\rangle \cong \mathrm{SL}(3,2)$.
    \end{lemma}
    \noindent {\it Proof} One checks that elements $(y_1,y_2,\dots,y_{123})\in \mathbb R^7$ having zero in one of these sets are permuted
    by $\mathrm{Aut}(F_3)$. Thus we obtain the (transitive) permutation representation $\Sigma_7$. The rest is a calculation.\qed 
    \medskip

Let $K_7=\ker(\Sigma_7)$. In   \cite {MKS} we can   find a presentation for Aut$(F_3)$, and since $K_7$ has index $168$ we can use the Reidemeister-Schreier process (as implemented in \cite {Ma}, for example),  to find generators for $K_7$.
These are
 \begin{align*}&
\tag {2.1} S,
U   S   U^{-1},
U^2,
P   U   S   U^{-1}   P,
P   U^2   P,
Q   U   S   U^{-1}   Q^{-1},
(Q^{-1}   U   Q)^2,
P   Q   U   S   U^{-1}   Q^{-1}   P,\\
&
P   Q   U^2   Q^{-1}   P,
P   Q^{-1}   U^2   Q   P,
(Q^{-1}   U   Q^{-1}   U   Q   U^{-1}   Q)^2,
(U   Q^{-1}   U   Q   U   Q^{-1}   U^{-1}   Q   U^{-1})^2,\\&
(U   Q^{-1}   U   Q^{-1}   U   Q   U^{-1}   Q   U^{-1})^2,
Q^{-1}   U   P   Q   U   Q   U^{-1}   Q^{-1}   U^{-1}   Q   U^{-1}   P,\\&
P   Q   U   Q^{-1}   U   P   U   Q   U^{-1}   Q^{-1}   P   U^{-1},
P   U   P   Q   U   P   U   P   U^{-1}   Q   U^{-1}   P,\\& 
P   Q^{-1}   U   Q^{-1}   U   P   U   Q   U^{-1}   Q^{-1}   P   U^{-1}   Q,
(Q   U   Q^{-1}   U   Q^{-1}   U   Q   U^{-1}   Q   U^{-1}   Q^{-1})^2.
 \end{align*}
 There are $18$ generators in this list of generators and $K_7/K_7' \cong \mathcal C_2^{18}$, so that we cannot generate $K_7$ by less than $18$ elements.

For generic $x,y,z \in \mathbb R$ define
\begin{align*}&
u_1=[x,y,0,z,0,0,0];
&u_2=[x,0,y,0,z,0,0];\\
&u_3=[0,x,y,0,0,z,0];
&u_4=[x,0,0,0,0,y,z];\\
&u_5=[0,x,0,0,y,0,z];
&u_6=[0,0,x,y,0,0,z];\\
&u_7=[0,0,0,x,y,z,0].
\end{align*}
Let $\mathcal U_i\subset \mathbb R^7, 1\le i\le 7,$ denote the $3$-dimensional subset determined by $u_i$, so that for example
$\mathcal U_1=\{(x,y,0,z,0,0,0) \in \mathbb R^7, x,y,z \in \mathbb R\}$.  From Lemma \ref {lems7} we see that the action of ${\rm Aut}(F_3)$ permutes the $\mathcal U_i$.

One also sees that $F(u_i), 1\le i\le 6$, is equal to $E(x,y,z)-1$, while $F(u_7)=E(x,y,-z)-1$. This shows that each $V(F) \cap \mathcal U_i=\partial \mathcal T_i, 1\le i \le 7,$ is a copy of   $\partial \mathcal T$. From Lemma \ref {lems7}, and the fact that $\mathcal U_i\cap V(F)=\partial \mathcal T_i$,
we see that the action of ${\rm Aut}(F_3)$ permutes the $\mathcal T_i$ with the action given by
Lemma \ref {lems7}.

Let $\Gamma_\mathcal U$ be the graph whose vertices are $\mathcal U_1,\dots,\mathcal U_7$, and where we have an edge $\mathcal U_i,\mathcal U_j$ exactly when $\dim (\mathcal U_i \cap \mathcal U_j)=1$.

\begin{lemma}\label {lem34}
The graph $\Gamma_\mathcal U$ is the Fano plane with seven points and seven lines as depicted in Figure 3 which is drawn on $S^2$ (each line is represented by a triangle). The group 
$\mathrm{Aut}(F_3)$ acts transitively on  $\mathcal U_1,\dots,\mathcal U_7$, giving an epimorphism $\mathrm{Aut}(F_3) \to {\rm SL}(3,2)$.

  \begin{figure}[hh] \label{cols}
\scalebox{.7}
{\includegraphics{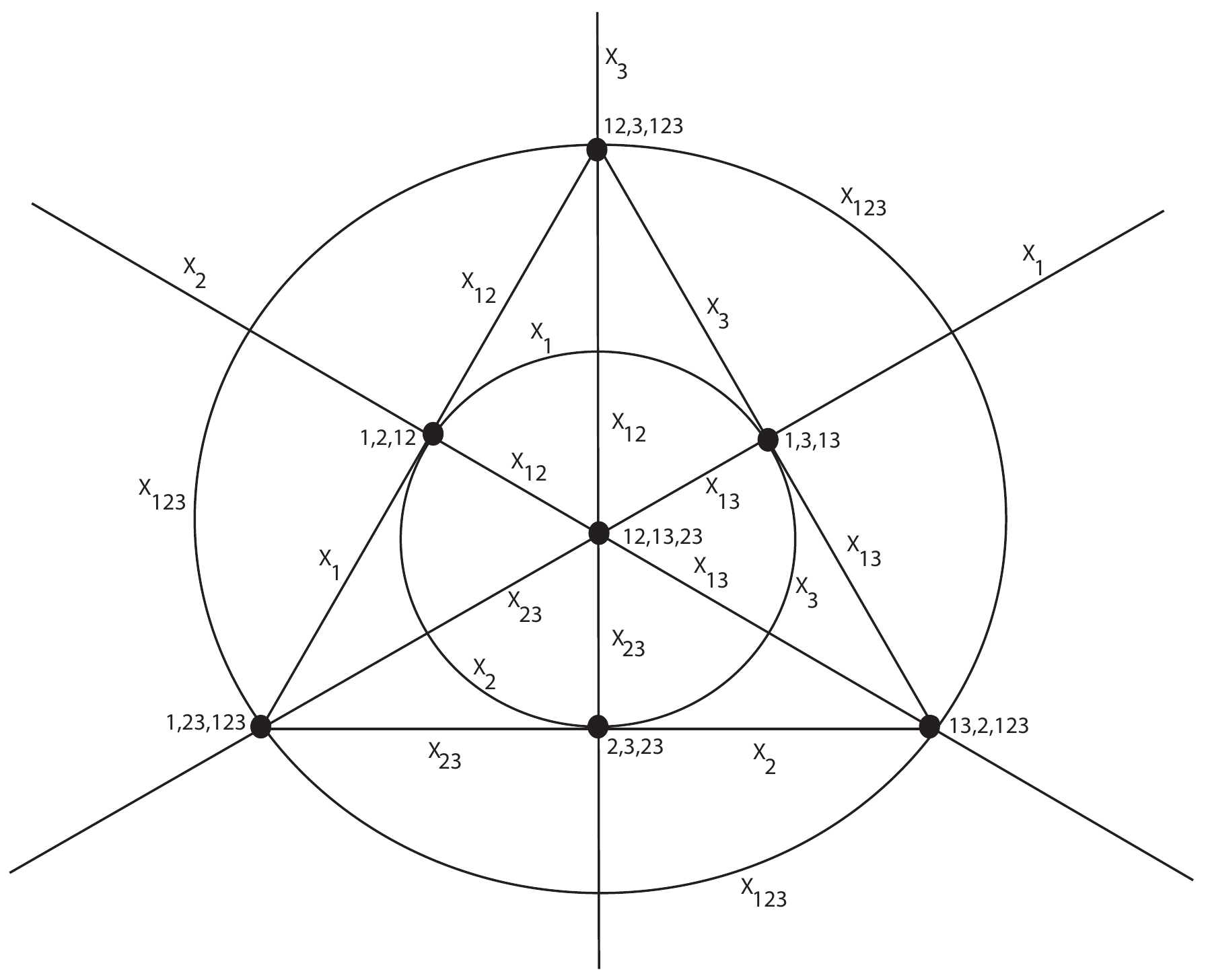}}
  \caption{The Fano plane.}
 \end{figure}

\end{lemma} 
\noindent {\it Proof} 
In Figure 3 we have drawn the graph $\Gamma_\mathcal U$, indicating each vertex $\mathcal U_i$ by the corresponding non-zero coordinates. For example, vertices labeled $12,3,123$ (corresponding to $\mathcal U_6$)  and $1,2,12$ (corresponding to $\mathcal U_1$)  are drawn with an edge between them since they share the coordinate $12$, and $\mathcal U_6 \cap \mathcal U_1$ is the $ x_{12}$-axis. The rest is clear.
\qed \medskip

We have now developed all the ideas necessary for us to be able to state our main result in the case $n=3$:

\begin {theorem} \label {mainthm}
If $p \in \mathcal P_3 \cap V(F)$, then we have one of the following:

\noindent (i) $p$ is on a coordinate axis;

\noindent (ii) there is some $1\le i\le 7$ such that $p \in \partial \mathcal T_i$;

\noindent (iii) $p \in {\rm Image}(\Pi_3)$;

\noindent (iv) the associated  group $\langle A_1, A_2, A_3\rangle$ is finite.

In particular, $\mathcal F_3=\mathcal P_3$.
\end{theorem}

We will say more about (ii) and (iii) below.

\section{Periodic points on $E_1$ and associated matrix groups
for points not on $E_1$}

  We first characterise those  points $p \in \mathcal
P_2$ which are on $E_1$, noting that the set $E_1$ contains
$\partial\mathcal T$.  

\begin{lemma}\label{l2.1} (i) A point $p \in
\partial \mathcal T$ is in $\mathcal P_2$ if and
only if there is $(\theta_1,\theta_2)^T \in \mathbb{ Q}^2/\mathbb{ Z}^2$
such that $\Pi_2(\theta_1,\theta_2)^T=p.$

(ii) Let $p \in E_1$.  Then $p \in \mathcal P_2$  if and
only
 $p \in \partial \mathcal T \cap \mathcal P_2$.
\end{lemma}

\begin{proof} (i)
First suppose that $\Pi_2(\theta_1,\theta_2)^T=p,$ where
$(\theta_1,\theta_2)^T \in \mathbb{ Q}^2/\mathbb{ Z}^2$.  This means that we
can write $\theta_1=s/n, \theta_2=t/n$ where $s,t,n \in \mathbb{ Z}, n
>0$.  Now given $\sigma \in {\rm Aut}(F_2)$ there is $u \in
\mathbb{ N}$ such that 
$\Phi_2(\sigma^u) \equiv I_2 \mod n$.  This means that the action of
$\Phi_2(\sigma^u)$ on $\mathbb{ T}^2$ fixes $(\theta_1,\theta_2)^T$.
Since the action of $\sigma$ on $\partial \mathcal T$ is determined by
the action of $\Phi_2(\sigma)$ on $\mathbb{ T}^2$ it follows that
$\sigma^u$ fixes $\Pi_2(\theta_1,\theta_2)^T=p$, as
required.

For the converse of (i) suppose that $(\theta_1,\theta_2)^T \in
\mathbb{ T}^2 \setminus (\mathbb{ Q}^2/\mathbb{ Z}^2)$; 
in fact we may assume that
$\theta_1 \notin \mathbb Q$ (the other case being similar).  Then from
the action of $\Phi_2(\sigma_2)$ and (\ref{eq1.1}) above we have
$\Phi_2(\sigma_2^k)(\theta_1,\theta_2)^T=(\theta_1,\theta_2-k\theta_1)^T$
and so $\Phi_2(\sigma_2^k)(\theta_1,\theta_2)^T \not \equiv
(\theta_1,\theta_2)^T \mod \mathbb{ Z}^2$ for all $k  \in \mathbb N$; this
shows that $(\theta_1,\theta_2)^T$ has an infinite orbit under the action of $\langle \sigma_2\rangle$ 
and concludes the
proof of (i).

(ii)  We now consider those points $p \in E_1 \cap  \mathcal P_2$ not on
$\partial \mathcal T$.  Any such point is on one of the cones which
meet $\partial \mathcal T$ at a point of $V$.  Since ${\rm Aut}(F_2)$ acts
transitively on these cones we may assume that $p$ is on the cone
$\mathcal C$ which meets $(1,1,1)^T \in V$.  It is easy to check that
this cone can be parameterised by the function
$$\Pi': \mathbb{ R}^2 \to \mathcal{ C};\quad \Pi'(t_1,t_2)^T=(\cosh
t_1,\cosh t_2,\cosh (t_1+t_2))^T.$$ The function $\Pi'$ is a $2$ to
$1$ cover branched at $(0,0)^T$, the point corresponding to
$(1,1,1)^T \in V$, and satisfies 
$\Pi'((t_1,t_2)^T)=\Pi'((-t_1,-t_2)^T)$.  As in (\ref{eq1.1}) the action of
${\rm Aut}(F_2)$ on $\mathcal C$ is induced from an action of
${\rm GL}(2,\mathbb{Z})$ on $ \mathbb{R}^2$: 
$$(\Pi'(t_1,t_2)^T)\alpha=\Pi'(\Phi_2(\alpha)(t_1,t_2)^T).$$
Thus the point $p \in \mathcal{C} \setminus V$ would correspond to
$(t_1,t_2)^T \in \mathbb{ R}^2 \setminus \{(0,0)^T\}$ and it is now
easy to see that either $\Phi_2(\sigma_1)$ or $\Phi_2(\sigma_2)$ has
infinite orbit on $(t_1,t_2)^T$, so $p$ has infinite orbit under
either  $\langle \sigma_1\rangle $ or $\langle \sigma_2\rangle$.  This proves (ii).
\end{proof}

Given $p=(x,y,z)^T \in \mathbb{ R}^3\setminus E_1,$ with associated matrices $A_1, A_2$,
defined up to conjugacy by $p$, we define the homomorphism
$$\mu=\mu(p):F_2 \to \langle A_1,A_2\rangle,\quad \mu(a_1)=A_1,\,\,\, \mu(a_2)=A_2.$$
 Given $\sigma \in {\rm Aut}(F_2)$ we obtain matrices $C=\mu(\sigma(a_1)),
D=\mu(\sigma(a_2))$.  We will denote $\mu(\sigma(a_1))$ by
$\sigma(A_1)$ and $\mu(\sigma(a_2))$ by $\sigma(A_2)$.  Then  $\langle A_1,A_2\rangle= 
\langle\sigma(A_1),\sigma(A_2)\rangle$.

The following shows that if $p \in \mathcal P_2$, then  $A_1, A_2$
have finite order under certain simple conditions.  

\begin{lemma}\label{l2.2} Let $p \in \mathcal{P}_2, p \notin
E_1,$ and also assume that $p$ does not lie on a coordinate axis.
Let $A_1=A_1(p), A_2=A_2(p)$ be the associated matrices.  Then $A_2$ has finite order, and for
any $\sigma \in {\rm Aut }(F_2)$ the matrix $\sigma(A_2)$ has finite
order.
\end{lemma}

\begin{proof} The result for arbitrary $\sigma$ will follow from the case $\sigma=1$. :et $p=(x,y,z)^T$. 

First note that
trace$\,(A_2)=2y$ where $ p=(x,y,z)^T$.  Now, since 
 $M \in {\rm SL}(2,\mathbb{C})$  has finite order if and only
if trace$(M)=2\cos 2\pi q$, for some $q \in \mathbb{Q}$ (except when this value is $\pm 2$, when we might have a parabolic, up to sign), we will next  show that $2y$ is of this form.

From (\ref{eq0.1})  we have: if $n \in \mathbb Z$, then
\begin{eqnarray*}
\begin{pmatrix} 0&0&1\\0&1&0\\-1&0&2y\end{pmatrix}^n\begin{pmatrix}
x\\y\\z\end{pmatrix}=\begin{pmatrix} x\\y\\z\end{pmatrix}\sigma_1^n;\\
\begin{pmatrix} 1&0&0\\0&2x&-1\\0&1&0\end{pmatrix}^n\begin{pmatrix}
x\\y\\z\end{pmatrix}=\begin{pmatrix} x\\y\\z\end{pmatrix}\sigma_2^n.
\end{eqnarray*}

Let $M=\begin{pmatrix} 0&0&1\\0&1&0\\-1&0&2y\end{pmatrix}$, so that $M\begin{pmatrix}
x\\y\\z\end{pmatrix} =\begin{pmatrix} x\\y\\z\end{pmatrix}\sigma_1$,  and let
$\lambda, 1/\lambda$ be the eigenvalues of the submatrix $M'=\begin{pmatrix} 0&1\\-1&2y\end{pmatrix}$.  Now trace$(M')=2y$,
where $y=$trace$(A_2)/2$.  Since $p \in \mathcal P_2$
we see that there is $n \in \mathbb N$ such that $(p)\sigma_1^n=p$
and so $(M')^n\begin{pmatrix} x\\z\end{pmatrix}=\begin{pmatrix} 
x\\z\end{pmatrix}$.  Now
if $x=z=0$, then $p=(x,y,z)^T $ is on a coordinate  axis.  Thus we may assume
that $(x,z)^T \ne (0,0)^T$.  It follows now that $1$ is an
eigenvalue of $(M')^n$,  showing that $\lambda^n=1$.  Thus $\lambda$ is
a root of unity and $2y=\lambda+1/\lambda=2\cos 2\pi q$, for some $q \in
\mathbb Q$.  We distinguish two cases:

{\bf Case 1: $\lambda \ne \pm 1$} Here the fact that
trace$(A_2)$=trace$(M')=2y=\lambda+1/\lambda=2\cos 2\pi q, q \in \mathbb
Q$, where
 $\lambda \ne 1/\lambda,$ shows that $M'$ is diagonalizable and of finite order, so that $A_2$ has finite order.

{\bf Case 2: $\lambda =\pm 1$} Here we may assume that $A_2=\begin{pmatrix}
a&0\\b&a\end{pmatrix}, a =\pm 1.$ There are two sub-cases: (i)
$b\ne 0$; and (ii) $b=0$.

If we have (i) $b \ne 0$, then we let $A_1=\begin{pmatrix}
u&v\\r&s\end{pmatrix},$  so that trace$(A_1A_2^m)=a^m(s+u+vabm)$, which takes
on infinitely many values for $m \in \mathbb N$ if $v \ne 0$,
contradicting $p\in \mathcal P_2$.  Thus $v=0$, from which it follows
that $A_1$ and $A_2$ have a common eigenvector and so $p \in E_1$, a
contradiction.

If we have (ii), then from trace$(A_1)=2x$ we get trace$(A_1A_2)=2ax$ and
$p=(x,a,ax) \in E_1$ since $a=\pm 1$.  This gives Lemma
\ref{l2.2}.
\end{proof}

\begin{corollary}\label{c2.3} Let $p \in \mathcal P_2$ and assume
that $p$ is not on level $1$ and not on a coordinate axis.
Let $A_1, A_2$ be the associated matrices.  Then
 each of $$ A_1A_2^h, A_2A_1^h, A_1(A_1A_2)^h, A_2(A_1A_2)^h, h \in \mathbb Z,$$ has 
finite order.
\end{corollary}

\begin{proof}
 It is easy to see that each of these
elements has the form $\sigma(A_2)$ for some $\sigma \in
{\rm Aut}(F_2)$, and so the result follows from Lemma \ref {l2.2}.  
\end{proof}

Recall that the points of $V$ are permuted transitively  by $\mathrm{Aut}(F_2)$. We will need:
\begin {lemma} \label {lemImPi2}
Let $p\in E_1 \cap \mathcal P_2 \setminus V$ correspond to the matrices $A_1,A_2$. Then

\noindent  (i) $A_1,A_2$ are conjugate to a pair of lower-triangular matrices.
 
\noindent  (ii) There is some element  in the ${\rm Aut}(F_2)$-orbit of $p$ that has $1$ as an entry.
\end{lemma}
\noindent {\it Proof} Let $p=(x_1,x_2,x_{12})$. (i) 
First  assume that one of $A_1, A_2, A_1A_2$ has distinct eigenvalues. Without loss of generality we assume that $A_1$ has eigenvalues $\lambda \ne 1/\lambda$. Then, by a conjugacy,  we can further assume $A_1=\begin{pmatrix} \lambda &0\\0&1/\lambda\end{pmatrix}, 
A_2=\begin{pmatrix} a&b\\c&d\end{pmatrix}.$ We calculate $x_1=\mathrm{trace}(A_1)/2, x_2=\mathrm{trace}(A_2)/2, 
x_{12}=\mathrm{trace}(A_1A_2)/2,$ and find that
$$E(x_1,x_2,x_{12})-1=\frac {1} {4\lambda^2} (1-\lambda^2)^2( 1-ad).$$
Since we have $E(x_1,x_2,x_{12})-1=0$ we thus have either (i) $\lambda=\pm 1$; or (ii) $ad=1$. If we have (ii), then $\det(A_2)=1$ gives $bc=0$, and we are done.
If we have  (i), then $\lambda=1/\lambda$, a contradiction. 

Thus we may now assume that each of $A_1, A_2, A_1A_2$ do not have distinct eigenvalues, so that $x_1,x_2,x_{12} \in \{\pm 1\}$.  But in this case we must have $p \in V$, a contradiction.

(ii) From (i) we can take $A_i=\begin{pmatrix}\lambda_i&0\\ c_i&\frac {1} {\lambda_i}\end{pmatrix},i=1,2$. We claim that  the  $\lambda_i$ are roots of unity.
Since $ p \in \mathcal P_2$ there $0<r_1<r_2<r_3$ such that
trace$(A_1^{r_1}A_2)$=trace$(A_1^{r_2}A_2)$ =trace$(A_1^{r_3}A_2)$. Let $1\le i\ne j\le 3$. Solving trace$(A_1^{r_i}A_2)$=trace$(A_1^{r_j}A_2)$
gives 
$$(\lambda_1^{r_i}-\lambda_1^{r_j})(\lambda_2^2\lambda_1^{r_i+r_j}-1)=0.$$
If $\lambda_1^{r_i}-\lambda_1^{r_j}=0$, then $\lambda_1$ is a root of unity. If not, then  we must have 
$\lambda_1^{r_i+r_j}=\lambda_2^{-2}$ for all such $i,j=1,2,3$. But then  
$\lambda_1^{r_1+r_2}=\lambda_1^{r_1+r_3}=\lambda_2^{-2}$, so that  $\lambda_1$ is again a root of unity. This proves the claim. 

Since the multiplicative group $\langle \lambda_1,\lambda_2\rangle\le \mathbb C^\times $ is a finite subgroup of a field
we see that it is cyclic, and so is  generated by $\lambda_1, \lambda_2$ or $\lambda_1\lambda_2$.
If it is generated by $\lambda_1$, then there is some $u \in \mathbb N$ such that ${\rm trace}(A_2A_1^u)=2$, and we have found an element of the orbit of $p$ that has $1$ as an entry.
If $\langle \lambda_1,\lambda_2\rangle$ is generated by $\lambda_2$ or $\lambda_1\lambda_2$, then we can similarly obtain condition (ii).\qed \medskip

We will have need of the following result   \cite [Algebraic Lemma, p. 101] {Dub}.
 An {\it admissible triple} is a triple $(a,b,c)\in \mathbb R^3$ where at most one of $a,b,c$ is zero.

\begin{lemma} \label {lemvinberg} Let $(a,b,c)\in \mathbb R^3$ be an admissible triple satisfying
$$a^2+b^2+c^2-2abc>1, \text { and } |a|,|b|,|c| \le 1. $$
Then there is  $\beta \in B_3$ such that the absolute value of one of the coordinates of $(a,b,c)\beta$ is greater than $1$.\qed\end{lemma}

\noindent {\it Proof of Theorem \ref {th1}} 

Let $p=(x,y,z) \in \mathcal P_2$ where $p \notin V$ is not on a coordinate axis.

If $p \in E_1$, then by Lemma \ref {lemImPi2}  we can assume that (up to a conjugacy) we have $A_1=\begin{pmatrix} 1&0\\c&1\end{pmatrix}, A_2=
 \begin{pmatrix} \lambda&0\\d&1/\lambda\end{pmatrix}$. Since the set of traces of the elements $A_1A_2^n, n \in \mathbb Z,$ is finite, we see that there are infinitely many $i,j \in \mathbb N, i \ne j,$ such that $\lambda^i+1/\lambda^i=\lambda^j+1/\lambda^j$; but then 
 $$0=\lambda^i+1/\lambda^i-\lambda^j-1/\lambda^j=(\lambda^{i+j}-1)(\lambda^i-\lambda^j)/(\lambda^{i+j}),$$
which shows that $\lambda$ is a root of unity. Thus we have case  (3) of Theorem \ref {th1}. 

Now assume that $p=(x,y,z) \in E_t, t>1$. Then by Lemma \ref {lemvinberg} we may assume that one of $x,y,z$ has absolute value greater than $1$.  In fact we may assume that $|y|>1$. But this contradicts Lemma \ref {l2.2} which shows that $|y|=|\cos(\theta)|$ for some $\theta \in \mathbb R$. 

We are left with the cases $p \in E_t, t<1$, which are dealt with in \cite {Dub}. \qed\medskip 


 \section {Periodic points in Image$(\Pi_3)$ and $V(\mathcal X)$}


As in the $n=2$ case we note that each element of $\Pi_3((\mathbb Q/\mathbb Z)^3)$ is in $\mathcal F_3$.

The analogue of Lemma \ref {l2.1} for the situation where  $n=3$ is

\begin{lemma}\label{l2.133} (i) A point $p \in
{\rm Image}(\Pi_3)$ is in $\mathcal P_3$ if and
only if there is $(\theta_1,\theta_2,\theta_3)^T \in \mathbb{ Q}^3/\mathbb{ Z}^3$
such that $\Pi_3(\theta_1,\theta_2,\theta_3)^T=p.$

(ii) Let $p  \in V(\mathcal X)\cap [-1,1]^7$.  Then $p \in \mathcal P_3$  if and
only if 
 $p \in {\rm Image}(\Pi_3) \cap \mathcal P_3$.
\end{lemma}
\noindent {\it Proof} The proof of  (i) is similar to the proof of Lemma  \ref {l2.1}   (i).

(ii)   So let $p=(x_1,x_2,\dots,x_{123}) ^T \in V(\mathcal X)\cap [-1,1]^7$. 
Then the fact that $|x_i|\le 1,i=1,2,3,$ means that we can write
$x_i=\cos( \theta_i),i=1,2,3$. Now $x_1^2+x_2^2+x_{12}^2-2x_1x_2x_{12}-1 \in \mathcal X$. Solving this equation for $x_{12}$ we get
$$x_{12}=\cos(\theta_1)\cos(\theta_2)\pm \sin(\theta_1)\sin(\theta_2)=\cos(\theta_1\mp \theta_2).$$
Thus we can change the sign of one of $\theta_1,\theta_2$ (if necessary) so that we have 
\begin{align*}\tag {4.1}
x_1=\cos( \theta_1), \,\,\, x_2=\cos( \theta_2),\,\,\, x_{12}=\cos( \theta_1+\theta_2).
\end{align*}

We similarly have $x_1^2+x_3^2+x_{13}^2-2x_1x_3x_{13}-1 \in \mathcal X$, which gives
$x_{13}=\cos(\theta_1\mp \theta_3).$ By changing the sign of $\theta_3$ if necessary, we may thus assume that
\begin{align*}
\tag {4.2} x_{13}=\cos(\theta_1+\theta_3).
\end{align*}

Now we also have  $x_2^2+x_3^2+x_{23}^2-2x_2x_3x_{23}-1 \in \mathcal X$
which gives $x_{23}=\cos(\theta_2\pm \theta_3)$. Of course, what we want is $x_{23}=\cos(\theta_2+ \theta_3)$.
So assume otherwise. Then we note that 
$$2{x_1} { x_2} x_3^{2}-  {  x_1} {  x_2}- {  
x_1} {  x_3} {  x_{23} }-  {  x_2} {  x_3} {  x_{13}}-
  x_3^{2}{  x_{12}}+  {  x_{12}}+{  x_{13}} {  x_{23}}
 \in \mathcal X.$$
 Now substituting (4.1), (4.2) and $x_{23}=\cos(\theta_2-\theta_3)$ into  this element gives 
$$-\sin(\theta_1)\sin(\theta_2)\sin^2(\theta_3),$$
which is non-zero for generic $\theta_i$, and so we must have
\begin{align*}
 x_{23}=\cos(\theta_2+ \theta_3).\end{align*}

Now, considering the element
$x_1^2+x_{23}^2+x_{123}^2-2x_1x_{23}x_{123}-1 \in \mathcal X$, we see that 
$x_{123}=\cos(\pm (\theta_1\pm (\theta_2+\theta_3))),$ while 
 considering the element
$x_2^2+x_{13}^2+x_{123}^2-2x_2x_{13}x_{123}-1 \in \mathcal X$, we see that 
$x_{123}=\cos(\pm (\theta_2\pm (\theta_1+\theta_3))).$ Thus the only possibility for reconciling these two equations  gives $x_{123}=\cos(\pm (\theta_1+ \theta_2+\theta_3)).$
\qed\medskip

\begin{lemma}\label{l27.5} (a) Suppose that we have a triple of matrices $A_1,A_2,A_3 \in \mathrm{SL}(2,\mathbb C)$ corresponding to the point $(x_1,x_2,x_3,x_{12},x_{13},x_{23},x_{123}) \in \mathcal P_3$, where one of $x_1,x_2,$ $x_3, x_{12},x_{13},x_{23},x_{123}$ is $\pm 1$, and the corresponding matrix is a non-identity parabolic (up to sign). Then 
$(x_1,x_2,x_3,x_{12},x_{13},x_{23},x_{123}) \in {\rm Image}(\Pi_3)$.

(b)
If one of $x_1,x_2,x_3,x_{12},x_{13},x_{23},x_{123}$ is $\pm 1$, and the corresponding matrix is the identity   matrix, then either 
$\langle A_1,A_2,A_3\rangle$ is a finite group or $(x_1,x_2,x_3,x_{12},x_{13},$ $x_{23},$ $ x_{123}) \in \partial \mathcal T_i$ for some $1\le i\le 7$.
\end{lemma}
\noindent {\it Proof}   Assume that one of $A_1,A_2,A_3,A_1A_2,A_1A_3,A_2A_3,A_!A_2A_3$ is $\pm K$, where $K$ is the identity or a parabolic.
 We first show that (after acting by some element of $\mathrm{Aut}(F_3)$)  we may assume that 
    $A_1=\pm K$ If $A_2=\pm K$ or $A_3=\pm K$, then using a power of $Q$ we reduce
 to the case where $A_1=\pm K$. If $A_1A_2=\pm K$, then, using the  automorphism determined by $[a_1,a_2,a_3] \mapsto [a_1a_2,a_2,a_3]$ will do this case.
   If we have $A_1A_3=\pm K$ or $A_2A_3=\pm K$, then we use a power of $Q$ to reduce to the case $A_1A_2=\pm K$, and this case follows. If we have $A_1A_2A_3=\pm  K$, then, using an automorphism, we replace $(A_1,A_2,A_3)$ by $(A_1A_2A_3,A_2,A_3)$, and we have done this last case.

Thus we may now suppose that $x_1=\pm 1$, any other case being similar by the above. Then (up to conjugacy) we have $A_1=\begin{pmatrix} a &0\\b&a\end{pmatrix}, a=\pm 1$. 
If we have (a), then   $b\ne 0$.
Now let $A_2=\begin{pmatrix} u&v\\w&x\end{pmatrix}$, so that   for all $n \in \mathbb Z$ we have
  trace$(A_1^nA_2)=a^n(u+x+abnv),$ which forces $v=0$, since $(x_1,x_2,x_3,x_{12},x_{13},$ $x_{23},x_{123}) \in \mathcal P_3$. Thus $A_1$ and $A_2$ are lower-triangular.
 Now by Corollary \ref {c2.3} we see that  trace$(A_2)=2\cos(2\pi p/q), p/q \in \mathbb Q$, and since $\det A_2=1$ it follows that $A_2=\begin{pmatrix} \lambda &0\\c&\overline {\lambda}\end{pmatrix}, \lambda=\cos (2\pi p/q)+i\sin(2\pi p/q).$
  One similarly shows that $A_3$ is    lower-triangular of this form. It follows that $(x_1,x_2,x_3,x_{12},x_{13},x_{23},x_{123}) \in {\rm Image}(\Pi_3)$. This proves (a).

 For (b) we now assume that $A_1=\varepsilon  I_2, \varepsilon=\pm 1$.  Then  
\begin{align*}\tag {4.3}\label {eq43} (x_1,x_2,x_3,x_{12},x_{13},x_{23},x_{123})=(\varepsilon,x_2,x_3,\varepsilon x_2,\varepsilon x_3,x_{23},\varepsilon x_{23}).
\end{align*}
 
 Now the subgroup of ${\rm Aut}(F_3)$ that fixes $a_1 \in F_3$ contains a copy of ${\rm Aut}(F_2)$ that acts on vectors of the   form (\ref {eq43})  just like ${\rm Aut}(F_2)$ acts on the triples $(x_1,x_2,x_{12}) \in \mathbb R^3$, so by Theorem \ref {th1} we have either
 
 \noindent (i) $\langle A_2,A_3\rangle$ is a finite group; or
 
 \noindent  (ii) $(x_2,x_3,x_{23})$ is on a coordinate  axis of $\mathbb R^3$; or
 
  \noindent (iii) $(x_2,x_3,x_{23}) \in {\rm Image}(\Pi_2)$.
  
  If we have (i), then $\langle A_1,A_2,A_3\rangle$ is a finite group, and we  are done.
  
    If we have (iii), then we write $x_2=\cos(\theta_2), x_3=\cos(\theta_3), x_{23}=\cos(\theta_2+\theta_3)$, so that if $\varepsilon=1$, then (\ref {eq43}) becomes
    \begin{align*}
    (x_1,x_2,x_3,x_{12},x_{13},x_{23},x_{123})=&(\cos(0),\cos(\theta_2),\cos(\theta_3),\cos(\theta_2),\cos(\theta_3),\\
    &\qquad \cos(\theta_2+\theta_3),\cos(\theta_2+\theta_3))\in {\rm Image}(\Pi_3).
    \end{align*}
    On the other hand, if $\varepsilon=-1$, then 
    \begin{align*}
    (x_1,x_2,x_3,x_{12},x_{13},x_{23},x_{123})=&(\cos(\pi),\cos(\theta_2),\cos(\theta_3),\cos(\pi+\theta_2),\cos(\pi+\theta_3),\\
    &\qquad \cos(\theta_2+\theta_3),\cos(\pi+\theta_2+\theta_3))\in {\rm Image}(\Pi_3).
    \end{align*}
This does case  (iii).

 If we have (ii), then without loss of generality we may assume that $(x_2,x_3,x_{23})=(0,0,x_{23})$.  Then from (\ref {eq43}) we have
 $$(x_1,x_2,x_3,x_{12},x_{13},x_{23},x_{123})=(\varepsilon,0,0,0,0,x_{23},\varepsilon x_{23}).$$
Then $E(\varepsilon,x_{23},\varepsilon x_{23})=1$ so that  
 $( \varepsilon,0,0,0,0,x_{23},\varepsilon x_{23}) \in \partial \mathcal T_4$.  \qed\medskip

\begin{lemma}\label{lemxgt1} Let $p=(x_1,x_2,x_3,x_{12},x_{13},x_{23},x_{123}) \in \mathcal P_3$.

\noindent (i) Suppose that one of $x_1,x_2,x_3,x_{12},x_{13},x_{23} ,x_{123}$ has absolute value greater than $1$. Then $p$ is on a coordinate  axis of $\mathbb R^7$.

\noindent (ii) Suppose that $\alpha \in {\rm Aut}(F_3)$ with $$(p)\alpha =(y_1,y_2,y_3,y_{12},y_{13},y_{23},y_{123}),$$ where one of  $y_1,y_2,y_3,y_{12},y_{13},y_{23},y_{123}$ has absolute value greater than $1$. Then $p$ is on a coordinate axis of $\mathbb R^7$, and so is $(x_1,x_2,x_3,x_{12},x_{13},x_{23},x_{123})$.
\end{lemma}
\noindent{\it Proof}
 For (i) we may assume (as in the proof of Lemma \ref {l27.5}) that $|x_1|>1$. Then using Theorem  \ref {th1}  applied to the point $q=(x_1,x_2,x_{12})\in \mathbb R^3$, we see that $q$ is on a coordinate  axis of $\mathbb R^3$, which shows that $x_2=x_{12}=0.$ The same argument applied to the triples $(x_1,x_3,x_{13})$ and $(x_1,x_{23},x_{123})$ yields $x_3=x_{13}=0, x_{23}=x_{123}=0$, which now shows that $p$ is on a coordinate  axis of  $\mathbb R^7$. This gives (i).

We note  that if $p\in \mathbb R^7$ is on a coordinate axis in $\mathbb R^7$, then so is any 
$(p)\alpha, \alpha \in {\rm Aut}(F_3)$. In fact, ${\rm Aut}(F_3)$ acts transitively on the axes,
so that (ii) follows from (i).\qed\medskip

\section {Binary dihedral groups}

In this section we investigate the connection between binary dihedral groups and the points of
$\mathcal P_3 \cap \mathcal U_i \cap V(F)=\mathcal P_3 \cap \partial \mathcal T_i, 1\le i\le 7$.

A {\it  binary dihedral group} $\mathrm{BD}_{2n}$ of order $4n$ has presentation
$$\langle a,b|a^{2n},b^4,a^n=b^2,a^b=a^{-1}\rangle.$$
It has a faithful representation in $\mathrm{SL}(2,\mathbb C)$  as
$$a=\begin{pmatrix}\varepsilon_{2n}&0\\0&\varepsilon_{2n}^{-1}\end{pmatrix};\,\,\,
b=\begin{pmatrix} 0&1\\-1&0\end{pmatrix}.$$ Here $\varepsilon_{2n}$ is a primitive $2n$th root of unity.  We will identify $\mathrm{BD}_{4n}$ with its image under this representation.
Thus when $A_1=a, A_2=b$ we have $x_1=\cos(2\pi k/2n),x_2=x_{12}=0.$ We call $a,b$   {\it standard generators of $\mathrm{BD}_{2n}$}.

We note that if $n$ is odd, then the matrices $a'=\begin{pmatrix}\varepsilon_{n}&0\\0&\varepsilon_{n}^{-1}\end{pmatrix},
b'=\begin{pmatrix} 0&1\\-1&0\end{pmatrix}$ still generate a binary dihedral group with standard generators $a(b')^2,b'$.

\begin{lemma} \label {lembdn}
(a) If $1\le i \le 7$ and   $p \in \mathcal P_3 \cap \mathcal U_i \cap V(F)$, and $p$ corresponds to the matrices $A_1,A_2,A_3$, then $\langle A_1,A_2,A_3\rangle$ is a binary dihedral group.

(b) If  $\langle A_1,A_2,A_3\rangle$ is a binary dihedral group, then the corresponding point $p \in \mathbb R^7$ is in $\mathcal P_3 \cap \mathcal U_i \cap V(F)$ for some  $1\le i\le 7$.
\end{lemma}

\noindent {\it Proof} Assume without loss of generality that $p\in \mathcal U_1 \cap V(F)$, so that we have
$x_3=x_{13}=x_{23}=x_{123}=0$. Since $x_3=0$ we see that $A_3$ has characteristic polynomial $x^2+1$, so that $A_3$ has order $4$ and is conjugate to $b$. We thus may assume that $A_3=b$. 

Since $p=(x_1,x_2,0,x_{12},0,0,0) \in \partial \mathcal T_1 \cap P_3$ we see  
that $E(x_1,x_2,x_{12})=1$ so that by Lemma \ref {l2.1}  we have 
$x_1=\cos(2\pi u/v), x_2=\cos(2\pi p/q)$. Let $A_1=\begin{pmatrix} a&b\\c&d\end{pmatrix}$, where we must have
$$a+d=2\cos(2\pi u/v),\quad ad-bc=1,\quad c-b=0.$$ The last equation comes from the fact that $x_{13}=0$.
Solving these equations gives
$$a=2\cos(2\pi u/v)-d,\quad b=c=\sqrt{ -d^2+2d\cos(2\pi u/v)  -1}.$$
Now any matrix of the form $R=\begin{pmatrix} r&s\\-s&r\end{pmatrix}, r,s \in \mathbb C,$ commutes with $b=A_3$, and so we can replace $A_1$ by $R^{-1}A_1R$. One then solves the $(1,2)$ entry of this matrix for $s$.
This has the effect of making  $R^{-1}A_1R$ a diagonal matrix with trace
 $2\cos(2\pi u/v)$. It follows that  $R^{-1}A_1R=\begin{pmatrix} \exp{2\pi i u/v}&0\\0& \exp{-2\pi i u/v}\end{pmatrix}.$ We assume that $A_1$ is this latter matrix.

Now let  $A_2=\begin{pmatrix} e&f\\g&h\end{pmatrix}$, so that we must have
$$e+h=2\cos(2\pi p/q),\quad eh-fg=1,\quad f=g,$$ where the last equation comes from the fact that
$x_{23}=0$. 

Then the condition $x_{123}=0$ gives $g=f\exp{4\pi i u/v}$. This, with $f=g$, gives either  (i) $f=g=0$; or (ii) $v=1,2$. If we have (i), then $A_2$ is also diagonal of finite order and we are done.
If we have (ii), then $x_1=\pm 1$, and $A_1=\pm I_2$; thus $\langle A_1,A_2,A_3\rangle$ is a binary dihedral group.
This proves (a), and (b) follows easily.\qed\medskip

 
  

\section{A review of a proof for $n=2$}

We briefly review one of the proofs of Theorem \ref {th1} from \cite {Dub}. 
In \cite {Dub}, 
for the case $n=2$,  the authors define 
$$g=\begin{pmatrix} 2 &2x_1&2x_{12}\\
2x_1&2&2x_2\\
2x_{12}&2x_2&2
\end{pmatrix},$$ 
showing that for $(x_1,x_2,x_{12})$ in the interior of $\mathcal T$ the matrix $g$ is positive definite. This uses Lemma \ref {lemvinberg} and the fact that
$$\det(g)=8(1-x_1^2-x_2^2-x_{12}^2+2x_2x_2x_{12})=8(1-E(x_1,x_2,x_{12})).$$
This then enables them to define a reflection subgroup $\langle R_1, R_2, R_3\rangle$.
Here
\begin{align*}&
R_1=\begin{pmatrix}
-1&-2x_1&-2x_{12}\\
0&1&0\\
0&0&1\end{pmatrix};\,\,\, R_2=\begin{pmatrix} 1&0&0\\-2x_1&-1&-2x_2\\0&0&1\end{pmatrix};\\&\qquad \qquad \qquad  R_3=\begin{pmatrix} 1&0&0\\
0&1&0\\
-2x_{12}&-2x_2&-1\end{pmatrix}.
\end{align*}
These reflections have the property that 
$$R_i^TgR_i=g, \text { for } i=1,2,3.$$
The authors of \cite {Dub}  then note that $\langle R_1,R_2,R_3\rangle$ is a subgroup of the orthogonal group (of $3 \times 3$ real matrices) determined by 
the positive definite form $g$.

They also show that the $x_1,x_2,x_{12}$ have the form $\cos (\pi p_u/q)$ for $p_u,q \in \mathbb Z,u=1,2,12$. This enables them to define the above matrices over $\mathbb Q[\zeta_q+1/\zeta_q]=\mathbb Q[\cos(\pi /q)]$. Using the Galois group $\mathcal G$ of $\mathbb Q[\cos(\pi /q)]/\mathbb Q$ it is clear that 
for $\gamma \in \mathcal G$ each $\gamma(R_i)$ fixes the form $\gamma(g)$. Take the direct sum 
$$\bigoplus (g)=\bigoplus_{\gamma \in \mathcal G} \gamma(g),$$ and consider the direct product group
$$\bigoplus(\langle R_1,R_2,R_3\rangle)=\bigoplus_{\gamma \in \mathcal G} \gamma(\langle R_1,R_2,R_3\rangle).$$ 

From the above it  
is clear that $\bigoplus(g)$ is a positive definite form on $\mathbb R^{3N}, N=|\mathcal G|$, fixed by the  group  $\bigoplus(\langle R_1,R_2,R_3\rangle)$; further, as the group $\langle R_1,R_2,R_3\rangle$
can be defined over the ring of integers $\mathcal O_q$ of $\mathbb Q[\zeta_q]$,
there is a lattice in $\mathbb R^{3N}$ that is fixed by $\bigoplus(\langle R_1,R_2,R_3\rangle)$ (namely $\mathcal O_q^{3N}$).
Thus $\bigoplus(\langle R_1,R_2,R_3\rangle)$ is a crystallographic group. Since $\bigoplus(\langle R_1,R_2,R_3\rangle)$ is a discrete subgroup of the compact   orthogonal group of $\mathbb R^{3N}, N=|\mathcal G|$, we see that $\bigoplus(\langle R_1,R_2,R_3\rangle)$ is finite, and so $\langle R_1,R_2,R_3\rangle$ is also finite.

This then limits the possibilities for $\langle R_1,R_2,R_3\rangle$ to a finite number of cases, using Coxeter's classification of finite reflection groups \cite {Cox}. 
For each case that has trace triples inside $\mathcal T$  we obtain a finite group $\langle A_1,A_2\rangle.$ This gives the result when $n=2$.
\medskip

\section {Reflection groups for $n=3$}

Given any symmetric $n \times n$ real matrix $g=(g_{i,j})$ with $2$s on the diagonal we can form reflections $R_1,\dots, R_n\in {\rm GL}(n,\mathbb R)$ where $R_i$ is the identity, except that the $i$th row of $R_i$ is 
$$-g_{i,1},-g_{i,2},\dots,-g_{i,i-1},-1,-g_{i,i+1},\dots,-g_{i,n}.$$ We say that {\it the $R_i$ are formed using $g$}.
 Then one can check that $R_1,\dots,R_n$ fix $g$ in the sense that $R_i^TgR_i=g$ for $i=1,\dots,n$.

In general, given matrices $A_1,\dots ,A_n \in \mathrm{SL}(2,\mathbb C)$ and letting $A_0=I_2,$  we can form the $(n+1)\times (n+1)$ symmetric matrix 
(with rows and columns indexed by $0,1,\dots,n$)
$$g_n=\left (\frac { {\rm trace}(A_iA_j^{-1})}{2}\right ).$$
Thus for example, when $n=2,$ we get the matrix
$$g_2=g_2(A_1,A_2)= \left( \begin {array}{ccc} 1&{  x_1}&{  x_2}\\ \noalign{\medskip}{
  x_1}&1&2 {  x_1} {  x_2}-{  x_{12}}\\ \noalign{\medskip}{  x_2}
&2 {  x_1} {  x_2}-{  x_{12}}&1\end {array} \right).
$$
 One checks that  
$$\det (g_2)=1-x_1^2-x_2^2-x_{12}^2+2x_1x_2x_{12}=1-E(x_1,x_2,x_{12}),$$
and it easily follows that $g_2$ is positive definite on the interior of $\mathcal T$.

For $n=3$, and $A_1,A_2,A_3$ associated to  $p=(x_1,x_2,x_3,x_{12},x_{13},x_{23},x_{123})$, we have:
$$g_3=g_3(p)= \left( \begin {array}{cccc} 1&{  x_1}&{  x_2}&{  x_{3}}
\\ \noalign{\medskip}{  x_1}&1&2 {  x_1} {  x_2}-{  x_{12}}&2 {
  x_{1}} {  x_3}-x_{13}\\ \noalign{\medskip}{  x_2}&2 {  x_1} 
{  x_2}-{  x_{12}}&1&2 {  x_{2}} {  x_3}-{  x_{23}}
\\ \noalign{\medskip}{  x_{3}}&2 {  x_{1}} {  x_3}-{  x_{13}}&2 {  
x_{2}} {  x_3}-{  x_{23}}&1\end {array} \right).$$
Our goal will be to show that $g_3$ is close to being positive-definite on $(-1,1)^7$.
Here 
\begin{align*}&
\det (g_3)=1-x_{23}^{2}
- x_{13}^{2}- x_{3}^{2}- x_{12}^{2}-
 x_2^{2}-  x_{1}^{2}+2 {  x_{23}} {  x_{3}} {  x_2}+2 {  
x_2} {  x_{1}} {  x_{12}}\\ &
\qquad +2 {  x_{3}} {  x_{1}} {  x_{13}}+  x_{1}^{
2} x_{23}^{2}+  x_{3}^{2}  x_{12}^{2}+ x_2^{2}  
x_{13}^{2}+4   x_{3}^{2}  x_2^{2} x_{1}^{2}+2 {  x_2} {
  x_{1}} {  x_{23}} {  x_{13}}\\&\qquad 
+2 {  x_{12}} {  x_2} {  x_{13}} {  
x_{3}}-4   x_{1}^{2}{  x_{23}} {  x_{3}} {  x_2}-4   x_{3}^{2}{
  x_{12}} {  x_2} {  x_{1}}-4   x_2^{2}{  x_{3}} {  x_{1}} {
  x_{13}}\\&\qquad +2 {  x_{1}} {  x_{23}} {  x_{3}} {  x_{12}}-2 {  x_{12}} {
  x_{13}} {  x_{23}}.
\end{align*}
\begin{lemma}\label{lempd} On $(-1,1)^7 \cap V(F)$ the function $\det (g_3)$ is positive, except on a $5$-dimensional subset of $V(F)$, where it is zero.
\end{lemma}
\noindent{\it Proof}
 We just need to note that
$$\det(g_3)+F=(2x_1x_2x_3+x_{123}-x_3x_{12}-x_1x_{23}-x_2x_{13})^2,$$
so that letting $$G=2x_1x_2x_3+x_{123}-x_3x_{12}-x_1x_{23}-x_2x_{13},$$ we see that  $V(F) \cap V(G)$ has dimension $5$, and the result follows.\qed \medskip

\noindent {\bf Remark}
For $n \ge 4$ the matrix $g_n$ is an $(n+1)\times (n+1)$ matrix where $n+1\ge 5$, and so by \cite [Theorem 2] {hk} we see that $\det(g_n)=0$. (The result of \cite {hk}   says that if $m_1,\dots,m_k,M_1,\dots,M_k\in {\rm SL}(2,\mathbb C), \varepsilon_1,\dots,\varepsilon_k \in \{\pm 1\},$ and $D=({\rm trace}(m_iM_j^{\varepsilon_i}))$ with $k \ge 5$, then $\det(D)=0$.) 
Thus there is no analogue of Lemma \ref {lempd} when $n \ge 4$, and so 
this method of proof will not work for $n \ge 4$. \medskip

Now given $A_1,A_2 ,A_3$ with $x_1={\rm trace}(A_1)/2$ etc, and $g=g_3, g=(g_{ij}),$ we define
\begin{align*}&
R_1=\begin{pmatrix}
-1&-2g_{12}&-2g_{13}&-2g_{14}\\
0&1&0&0\\
0&0&1&0\\0&0&0&1\end{pmatrix};\,\,\,
R_2=\begin{pmatrix} 1&0&0&0\\-2g_{21}&-1&-2g_{23}&-2g_{24}\\0&0&1&1\\0&0&0&1\end{pmatrix};\\&
R_3=\begin{pmatrix}
1&0&0&0\\
0&1&0&0\\
-2g_{31}&-2g_{32}&-1&-2g_{34}\\0&0&0&1\end{pmatrix};\,\,\,
R_4=\begin{pmatrix}
1&0&0&0\\
0&1&0&0\\
0&0&1&0\\-2g_{41}&-2g_{42}&-2g_{43}&-1\end{pmatrix}.
\end{align*} 
Thus  $R_1,R_2,R_3,R_4$ are formed  using $2g_3$, a real symmetric matrix with $2$s on the diagonal. 
Thus $g_3$ is fixed by any element of the group
$\langle R_1,R_2,R_3,R_4\rangle$: for $R \in \langle R_1,R_2,R_3,R_4\rangle$ we have
$R^Tg_3R=g_3$. 

\begin {proposition} \label {propg3posdef} Assume that  $p \in V(F) \cap (-1,1)^7 \cap \mathcal P_3, p \notin V(G)$. Let the associated matrices be $A_1,A_2,A_3$.
Then  either 

\noindent (i) $p \in \mathcal U_i$ for some $1\le i\le 7,$ or 

\noindent (i) $\langle A_1,A_2,A_3\rangle$ is finite, or  

\noindent (iii) $p \in {\rm Image}(\Pi_3),$ or

\noindent (iv)  the matrix $g_3$ is positive-definite at $p$.
\end{proposition}
\noindent {\it Proof} For  positive-definiteness we just check the usual  conditions (Sylvester's criterion) on the determinants of the various $k \times k$ principal submatrices of $g_3$, $k=1,2,3,4$. The case $k=1$ is clear.  

When $k=2$ the principal $2\times 2$ sub-matrix is $\begin{pmatrix} 1&x_1\\x_1&1\end{pmatrix}$, which has determinant $1-x_1^2>0$, since $p \in (-1,1)^7$. 

When $k=3$ the principal $3\times 3$ sub-matrix with rows and columns with indices $0,1,2$ has determinant $1-E_1(x_1,x_2,x_{12})$. 
Now  if $1-E_1(x_1,x_2,x_{12})<0$, then
$E(x_1,x_2,x_{12})>1$, and  using Lemma \ref {lemvinberg} we see that there is a point $p'$ in the $\mathrm{Aut}(F_3)$-orbit of $p$ such that $p'$ has a coordinate that 
is greater than $1$ in absolute value. We can assume that $p'=(p_1,p_2,\dots,p_{123})$ where $|p_1|>1$. But 
Lemma \ref {l2.2} shows that $A_1$ has finite order, so that $|p_1|\le 1$, a contradiction.

On the other hand,  if $1-E_1(x_1,x_2,x_{12})=0$, then Lemma \ref {lemImPi2} tells us that $(x_1,x_2,x_{12})$ has an element  in its ${\rm Aut}(F_2)$-orbit that has  a $1$ as one of its entries. It follows that $p$ has an element in its  ${\rm Aut}(F_3)$-orbit that has  a $1$ as one of its entries. Thus by Lemma \ref {l27.5} we see that either $p \in \mathcal U_i$, for some $1\le i\le 7,$ or $\langle A_1,A_2,A_3\rangle$ is finite. This does the case $k=3$.

 For the $k=4$ case we note that since $p \notin V(G)$, Lemma \ref {lempd} gives  $\det(g_3)>0$.\qed\medskip

The above shows that either we have what we want for Theorem \ref {mainthm}  (namely that $p \in \mathrm{Image}(\Pi_3)$, or $p$ is on an axis, or $\langle A_1,A_2,A_3\rangle$ is finite, or $p \in \mathcal T_i$), or $G(p)=0$, or
$g_3(p)$ is positive-definite.

Now we show how  to deal with the fact that on the set $V(F) \cap V(G)$ the form $g_3$ is not positive definite.

We recall
the
{\it Cholesky decomposition}: the (real) matrix $M$ is positive definite if and only if there exists a (real) unique lower triangular matrix $L$, with real and strictly positive diagonal elements, such that $M = LL^T$.
For the matrix $M=g_3$ we find that $L=L(p)$ has the form
\begin{align*}
\tag {7.1}
\begin{pmatrix} 1&0&0&0\\
x_1&\sqrt{1-x_1^2}&0&0\\
x_2&\frac {x_1x_2-x_{12}}{\sqrt{1-x_1^2}}&
\frac {\sqrt { (1-x_1^2)(1-E(x_1,x_2,x_{12}) )} } { x_1^2-1  }&0\\
x_3&\frac {x_1x_3-x_{13}}{\sqrt{1-x_1^2}}& 
\frac {W   } { \sqrt { (1-x_1^2)(1-E(x_1,x_2,x_{12}) )  }}
&
\frac { \sqrt{ (1-E(x_1,x_2,x_{12}) ) \det (g_3)        }   } {E(x_1,x_2,x_{12})-1   }
\end{pmatrix},
\end{align*} where $W=-x_1x_2-x_1x_3x_{12}+x_{12}x_{13}-x_1x_2x_{13}+2x_1^2x_2x_3-x_1^2x_{23}+x_{23}.$

We note that each element whose square root is taken is non-negative, since we can assume that $ x_1^2<1, E(x_1,x_2,x_{12})<1$ and $\det (g_3) \ge 0$. This again reduces the question of whether $g_3$ is positive definite to the question of whether $G$ is non-zero.

Now we see what happens if we replace $A_1,A_2,A_3$ by $\alpha(A_1),\alpha(A_2),\alpha(A_3), \alpha \in \mathrm{Aut}(F_3)$. Suppose that 
$$(p)\alpha=(x_1,x_2,x_3,x_{12},x_{13},x_{23},x_{123})\alpha=
(y_1,y_2,y_3,y_{12},y_{13},y_{23},y_{123}).$$  Then $L((p)\alpha)$ is given by (7.1) with $y_I$ replacing $x_I$. Of course  $\langle A_1,A_2,A_3\rangle=\langle \alpha(A_1),\alpha(A_2),\alpha(A_3)\rangle$. Further, the positivite-definitenes of $\det g_3((p)\alpha)$ now depends on whether 
$G((p)\alpha)=G(y_1,y_2,y_3,y_{12},y_{13},y_{23},y_{123})$ is non-zero or not. If it is non-zero, then  $g_3((p)\alpha)$ is positive-definite, which is what we want.

Thus the `bad' case is if  $G((p)\alpha)$ is zero for all $\alpha \in \mathrm{Aut}(F_3)$.
We thus look at the smallest ideal that contains $G$ and $F$ and is ${\rm Aut}(F_3)$-invariant. A Magma computation
    shows that this ideal is $\mathcal X$. Thus in this (bad) case $p \in \mathcal X$, and we are done by Lemma \ref {l2.133}.
%
    Thus we have part (a) of:
    
\begin {proposition} \label {propg3posdef2}  Assume that $p \in \mathcal P_3$, where $p \in V(F) \cap (-1,1)^7,$ $p$ is not on a coordinate axis or in any of the $\mathcal \partial T_i, 1\le i\le 7$, and   $p \not\in {\rm Image}(\Pi_3)$. Let $R_1,R_2,R_3,R_4$ be the reflections  formed from $2g_3(p)$.
Then

\noindent (a)  the  matrix $g_3(p)$ is positive-definite at $p$;

\noindent (b)   $\langle R_1,R_2,R_3,R_4\rangle $ is a finite reflection group.
\end{proposition}
\noindent {\it Proof} Part (b) follows from (a) as in the proof of the $n=2$ case outlined in $\S 6$.\qed\medskip

Now to complete the proof of our  result for $n=3$ we need only consider the possibilities for  $R_1,R_2,R_3,R_4$ and $G=\langle R_1,R_2,R_3,R_4\rangle$
using Coxeter's classification of finite groups generated by reflections \cite {Cox}.

First note that each $R_i$ is determined by a vector $v_i \in \mathbb R^4$ such that 
$$R_i(v)=v-2\frac {(v,v_i)}{(v_i,v_i)}v_i,\,\,\,\, 1\le i\le 4,\,\,\,\, v \in \mathbb R^4.$$
 Here $(\cdot,\cdot)$ is the form determined by $g_3$. We further note that distinct  reflections $R_i, R_j$ with vectors $v_i,v_j$ (respectively) commute if and only of $(v_i,v_j)=0$.

Now from the theory of Coxeter groups \cite {hum} the decomposition of $G$ as a direct sum is determined by the components of the graph $\Gamma(G)$ whose vertices are $v_1,\dots,v_4$ and where we have an edge $v_i,v_j$ whenever $(v_i,v_j) \ne 0$.
Thus any such finite reflection group $G$ decomposes as a direct product of irreducible reflection groups, one for each component of $\Gamma(G)$.

Coxeter's classification \cite {Cox} shows that there are five of these groups that are irreducible. 

If $G$ is a direct product of two irreducible finite reflection groups, then we have either

\noindent (ai) $G=G_3\times G_1$; or (aii)  $G=G_2\times G_2'$.

 Here $G_i,G_i'$ are irreducible finite reflection groups of degree $i$, and they are generated by $i$ of the reflections $R_1,R_2,R_3,R_4$. Clearly, there is only one such group of degree $1$. Similarly, if  $G$ is a direct product of three irreducible finite reflection groups, then we have 

\noindent(bi) $G=G_2\times G_1\times G_1$. 

Lastly, if  $G$ is a direct product of four  irreducible finite reflection groups, then we have 

\noindent (ci) $G=G_1\times G_1\times G_1\times G_1$.
\medskip

If $G$ is irreducible, then $G$ has type $A_4, B_4, D_4, F_4, H_4$ of orders $120, 384, 192, $ $ 1152, 
14400$ (respectively). In these cases the order of $R_iR_j, 1\le i,j \le 4$ is in $\{1,2,3,4,5\}$, so that 
the values of $x_1,x_2,\dots,x_{123}$ are $\cos(\frac {\pi k}{m})$, where $m=1,2,3,4,5$ and $0\le k\le m$; see Corollary \ref {corvals}.  There are eleven such cosine values, including $\pm 1$. 

This reduces the checking of this (irreducible) case to a finite number of cases. We note that if one of 
 $x_1,x_2,\dots,x_{123}$ is $\pm 1$, then this situation  is covered by Lemma \ref {l27.5}.
 Thus there are now nine possible values for  $x_1,x_2,\dots,x_{123}$ (the $\cos(\frac {\pi k}{m})$, where $m=1,2,3,4,5$ and $0\le i\le m$, that are not $\pm 1$). 
 
 Checking (using \cite {Ma}) we find that the only situation where the ${\rm Aut}(F_3)$ orbit of  $(x_1,x_2,\dots,x_{123})$  is finite is when the orbit 
has size one of $168,3360,336,520,112$, with the corresponding group $G$ having order $32,48,100,192,36$ (respectively).
 However one checks that each such orbit has a point in it that contains a $1$, so that
these cases are also  covered by Lemma \ref {l27.5}.  (We note that the groups $G$ so found are not necessarily irreducible, but this does cover all such irreducible $G$.)

Now if we have a group $G$ of type (ai), then the possibilities for $G_3$ are the Coxeter groups of type $A_3, B_3, H_3$ of orders $24, 48, 120$ (respectively).
The possibilities for the values of $x_1,x_2,\dots,x_{123}$ are the same as in the irreducible case just considered, and so these cases are covered by the calculations for the irreducible case above.

Now, in considering cases (aii) and (bi), the possibilities for $G_2$ are the Coxeter groups of type $I_2(k)$ i.e. the dihedral groups $D_{2k}$ of order $2k$.

Thus if we have (aii) or (bi) with $\langle R_i,R_j\rangle \cong D_{2u},u\ge 3,$ then $G=\langle R_i,R_j\rangle \times H$, where $H$ is generated by the $R_k, k \ne i,j.$ 
Thus we have $(v_i,v_k)=(v_j,v_k)=0$ for $k \ne i,j$. This shows that four of the $x_1,x_2,x_3,2x_1x_2-x_{12},2x_1x_3-x_{13},2x_2x_3-x_{23}$ are  zero.

For example, suppose that $i=1,j=2$. Then we see that
$$x_2=x_3=2x_1x_2-x_{12}=2x_1x_3-x_{13}=0,$$ which gives $x_2=x_3=x_{12}=x_{13}=0$. This   shows that $p=(x_1,0,0,0,0,x_{23},x_{123}) \in \mathcal U_4.$
Similarly, if $i=1,j=3$, then we get $x_1=x_3=x_{12}=x_{23}=0$, which gives
$p=(0,x_2,0,0,x_{13},0,x_{123})\in \mathcal U_5$. If $i=1,j=4$, then we get
$x_1=x_2=x_{13}=x_{23}=0$, giving $p=(0,0,x_3,x_{12},0,0,x_{123}) \in \mathcal U_6$. The rest of the cases are similar.
 This concludes consideration of (aii) and (bi).

The last case, (ci), is not allowed, since $p \notin V(F)$ in this case. This concludes the proof of Theorem \ref {mainthm}.\qed\medskip

We note that in Theorem \ref {mainthm}, any $p \in \mathcal P_3 \cap V(F)$ of type (i) has its only non-zero entry equal to $\pm 1$ (since $p \in V(F)$), and so  is in some $\partial \mathcal T_i,1 \le i\le 7$; it is thus  covered by type (ii).
Further, from the results of $\S 5,$ we see that any $p \in \mathcal P_3$ of type (ii) has
$\langle A_1,A_2,A_3\rangle$ a binary dihedral group, and so is covered by type (iii). Thus we now have

\begin {corollary} \label {mainthm2}
If $p \in \mathcal P_3 \cap V(F)$, then $p \in \mathcal F_3$, and we have one of the following:

\noindent (i)
$p \in {\rm Image}(\Pi_3)$; 

\noindent (ii) the associated  group $\langle A_1, A_2, A_3\rangle$ is finite.\qed 
\end{corollary}

\section {The proof of the  $n>3$ cases} 

In general we have the polynomial ring $R_n=\mathbb Q[x_1,x_2,\dots,x_{12\dots n}]$. We let $\mathfrak T_n$ be the {\it trace ideal} of $R_n$ consisting of elements $x=x( x_1,x_2,\dots,x_{12\dots n}) \in R_n$ such that for all $A_1,A_2,\dots,A_n \in \mathrm{SL}(2,\mathbb R)$ we have 
$$x(\mathrm{trace}(A_1)/2,\mathrm{trace}(A_2)/2, \dots,\mathrm{trace}(A_1A_2\dots A_n)/2)=0.$$
The {\it trace ring} is $R_n/\mathfrak T_n$, and there is a well-defined action of $\mathrm{Aut}(F_n)$ on $R_n/\mathfrak T_n$. The ideal $\mathfrak T_n$
determines a subset $V(\mathfrak T_n)$ of $\mathbb R^{2^n-1}$ that is thus invariant under the action of  $\mathrm{Aut}(F_n)$, and on which there is an action of  $\mathrm{Aut}(F_n)$. 

In Lemma \ref {lemact} we have shown that there is a well-defined action of $\mathrm{Aut}(F_3)$ on all of $\mathbb R^7$ given by the action on the traces; in the $n=2$ case there was also a well-defined action of  $\mathrm{Aut}(F_2)$ on $\mathbb R^3$. 
In order to obtain an action of $\mathrm{Aut}(F_n), n>3,$ (also given by the action on the traces) we will have to restrict to the action on the points 
$(x_1,x_2,\dots,x_{123\dots n}) \in \mathbb R^{2^n-1}$ where there are $A_1,A_2,\dots,A_n \in \mathrm{SL}(2,\mathbb R)$ such that
$x_1=\mathrm{trace}(A_1)/2,x_2=\mathrm{trace}(A_2)/2, \dots,x_{12\dots n}=\mathrm{trace}(A_1A_2\dots A_n)/2.$
The point here is that if we define the action of $\alpha \in \mathrm{Aut}(F_n)$ on 
$\mathbb Q[x_1,x_2,\dots,x_{12\dots n}]^{2^n-1},$ using the trace identities, then for 
$p=(x_1,x_2,\dots,x_{12\dots n})\in  \mathbb Q[x_1,x_2,\dots,x_{12\dots n}]^{2^n-1},$ with $(p)\alpha=(y_1,y_2,\dots,y_{12\dots n})$, we see that $x_I-y_I$ will be in the trace ideal $\mathfrak T_n$, meaning that $x_I(A_1,A_2,\dots,A_n)= y_I(A_1,A_2,\dots,A_n)$ for all such $I$.

\begin{lemma} \label {lemis1}
If $p \in \mathcal P_3$ has finite  associated  matrix group, then the $\mathrm{Aut}(F_3)$ orbit of $p$ contains a point with an entry that is $1$.
\end{lemma}
\noindent {\it Proof}  
If the group $\langle A_1,A_2,A_3\rangle$ is cyclic, then it is generated by one of  $A_1, A_2, $ $A_1A_2, A_1A_3, A_2A_3, A_1A_2A_3$. If it is generated by $A_1$, then there is some $m \in \mathbb N$ such that $A_2A_1^m=I_2$, and we are done. The other cases are similar.

If the group $\langle A_1,A_2,A_3\rangle$ is binary dihedral with standard generators $a,b$ (as in $\S 5$), then one of the $A_i$ has the form $a^kb$ or $a^kb^{-1}$ (for some $k$). By an automorphism we can assume that $A_1$ has the form $a^kb$. By a further automorphism we may assume that $A_2,A_3$ are each powers of $a$, so that $\langle A_2,A_3\rangle$ is a cyclic group $\langle a^u\rangle$. Thus this case now follows as in the cyclic group case just considered.

Other than the cyclic and binary dihedral groups, there are three  finite groups to consider:
For BT$_{24}$ (or a subgroup of BT$_{24}$) we find that there are $7$ orbits of sizes 
$1, 7, 13, 14, 28, 91, 520$. 
 For BO$_{48}$ (or a subgroup of BO$_{48}$)  we find that there are $11$ orbits of sizes 
 $ 1, 7, 13, 14, 28, 91, 112, 168, 224, 520, 3360$. 
 For BI$_{120}$ (or a subgroup of BI$_{120}$) we find that there are $12$ orbits of sizes
  $1, 7, 13, 14, 28, 62, 91, 112, 336, 434, 520,$ $ 26688$. In every case each orbit has an element that contains $1$ as an entry.
\qed\medskip 

\noindent {\bf Remark} We gave the sizes of the orbits for subgroups of $\mathrm{BI}_{24}$ in the proof of the  last result. One finds that the permutation representations of $\mathrm{Aut}(F_3)$ that one obtains in each case have orders
\begin{align*}& 1,\quad  2^3\cdot3\cdot 7,\quad  2^4\cdot  3^3\cdot 13, \quad  2^{10}\cdot 3\cdot 7, \quad  2^{11}\cdot 3 \cdot 7,\quad 
 2^ 7\cdot 3^4\cdot 7\cdot 13, \\& \qquad 2^{190} \cdot 3^{ 17} \cdot 5^{13}\cdot 7\cdot 13,
 \end{align*} (respectively).
 
 For the orbits   of subgroups of $\mathrm{BO}_{48}$ one similarly  obtains  the following permutation group orders:
\begin{align*}&
 1,\quad  2^3\cdot3\cdot 7,\quad   2^4\cdot  3^3\cdot 13, \quad  2^{10}\cdot 3\cdot 7, \quad 
 2^{11}\cdot 3 \cdot 7,\quad  2^ 7\cdot 3^4\cdot 7\cdot 13,\\& 2^{29}\cdot 3^8\cdot 7,\quad  2^{10}\cdot3\cdot 7,
 \quad  2^{19}\cdot3\cdot 7,\quad  2^{190} \cdot 3^{ 17} \cdot 5^{13}\cdot 7\cdot 13,\\& 2^{1037}\cdot 3^{64}\cdot 5^{28}\cdot 7^{29},
\end{align*} (respectively).

 For the orbits   of subgroups of $\mathrm{BI}_{120}$ one obtains  the following permutation group orders: 
\begin{align*}&
 1,\quad  2^3\cdot3\cdot 7,\quad   2^4\cdot  3^3\cdot 13, \quad  2^{10}\cdot 3\cdot 7, \quad
  2^{11}\cdot 3 \cdot 7,\quad 2^2\cdot 3\cdot 5^3\cdot 31,\quad 2^7\cdot 3^4\cdot 7\cdot 13,\\&
  2^{29}\cdot 3^8\cdot 7,\quad 2^{29}\cdot 3^8\cdot 5^7\cdot 7,\quad 2^8\cdot 3^2\cdot 5^3\cdot 7\cdot 31,\quad 
  2^{190}\cdot 3^{17}\cdot 5^{13}\cdot 7\cdot 13,\\&
  2^{6671}\cdot 3\cdot 7 \cdot (1668!).
\end{align*} (respectively). These calculations were accomplished using Magma.
 
 \medskip

One says that the $n$-tuple $A_1,A_2,\dots,A_n$ of elements of $\mathrm{SL}(2,\mathbb C)$  {\it  is conjugate to an upper-triangular $n$-tuple}
if there is $g \in \mathrm{SL}(2,\mathbb C)$ such that each of $A_1^g,A_2^g,\dots,A_n^g$ is an upper-triangular matrix.

\begin{lemma} \label {lemut} Let $p \in \mathcal P_3$ correspond to the triple of matrices  $A_1,A_2,A_3\in \mathrm{SL}(2,\mathbb C)$.
Then the triple
 $(A_1,A_2,A_3)$ is conjugate to a triple of upper-triangular matrices if and only if  $p\in \mathrm{Image}(\Pi_3)$.
\end{lemma}
\noindent {\it Proof}  If  $(A_1,A_2,A_3)$ is conjugate to a triple of upper-triangular matrices, then we may assume that
$A_1=\begin{pmatrix}a_1&b_1\\0&1/a_1\end{pmatrix},
A_2=\begin{pmatrix}a_2&b_2\\0&1/a_2\end{pmatrix},
A_3=\begin{pmatrix}a_3&b_3\\0&1/a_3\end{pmatrix}.$ Since $p \in \mathcal P_3$ it follows that
$A_1,A_2,A_3$ have finite order, or are $\pm K$, where $K$ is an upper-triangular  parabolic (see $\S 4$). Thus the eigenvalues $a_i, 1/a_i$ of $A_i$  are $\zeta_i,1/\zeta_i$, roots of unity.
Thus there are $\theta_i \in \pi \mathbb Q,i=1,2,3,$ such that $x_i=\cos (\theta_i),x_{ij}=\cos(\theta_i+\theta_j), x_{ijk}=\cos(\theta_i+\theta_j+\theta_k),$ and it follows that $p \in \mathrm{Image}(\Pi_3)$. 

The converse is given by \cite  [Corollary 2.11] {Fl}.
\qed\medskip

Given $A_1,A_2,\dots,A_n,n\ge 3$ and $1\le i<j<k\le n$, the triple $A_i,A_j,A_k$ determines a $7$-tuple of traces $(x_1,x_2,\dots,x_{123})$
on which  we can evaluate $F$; denote the resulting value by $F_{i,j,k}$.
We immediately see that $F_{i,j,k}=0$ for all such $i,j,k$. Thus we may apply Corollary \ref {mainthm2}  to the triple $A_i,A_j,A_k$.

For the main result ($n>3$) we need to define $\Pi_n:\mathbb R^n \to \mathbb R^{2^n-1},n>3,$ analogously to $\Pi_2,\Pi_3$:
\begin {align*}
\Pi_n(t_1,t_2,\dots,t_n)=(\cos(2\pi t_1),&\cos(2\pi t_2),\dots,\cos(2\pi t_n),\cos(2\pi (t_1+t_2),\\&\dots,\cos(2\pi(t_1+t+2+\dots t_n)).
\end{align*}

The proof of our main result proceeds by induction on $n\ge3$, where Corollary \ref {mainthm2} gives the $n=3$ case. The result that we prove is

\begin{theorem} \label {thmmainngt3}
If $p \in \mathcal P_n, n \ge 3,$ corresponds to the matrices $A_1,\dots,A_n$, then $p \in \mathcal F_n$, and we have one of the following:

\noindent (i)
$p \in {\rm Image}(\Pi_n)$; 

\noindent (ii) the associated  group $\langle A_1, A_2, \dots,A_n\rangle$ is finite.

In particular, $\mathcal P_n=\mathcal F_n$. 
\end{theorem}
\noindent {\it Proof}
So let $p \in \mathcal P_n$ correspond to the matrices $A_1,\dots,A_n$.  
If for each distinct triple $1 \le i,j,k\le n$ the group $\langle A_i,A_i,A_i\rangle$ is conjugate to an upper-triangular triple, then \cite [Theorem 2.7] {Fl} shows that the $n$-tuple $A_1,A_2,\dots,A_n$ is conjugate to an upper-triangular $n$-tuple.

Now suppose that some triple of the matrices $A_1,\dots,A_n$ is not conjugate to an upper-triangular $n$-tuple. By a permutation action of $\mathrm{Aut}(F_n)$ (and for ease of notation)  we may assume that this triple is $A_1,A_2,A_3$.  Then by Corollary \ref {mainthm2} we see that $\langle A_1,A_2,A_3\rangle$ is a finite group. Then by Lemma \ref {lemis1}  we can assume that one of $x_1,x_2,x_3,x_{12},x_{13},x_{23},x_{123}$ is $1$. Thus there is some
$\alpha \in \mathrm {Aut}(F_3) \le \mathrm {Aut}(F_n)$ such that $(p)\alpha=(1,x_2',x_3',x_{12}',x_{13}',x_{23}',x_{123}',\dots)$, which implies that $A_1=I_2$, since 
$\langle A_1,A_2,A_3\rangle$ is  finite. 
But with $A_1=I_2$, the result now follows by induction.\qed

\end{document}